\documentclass{article}

\usepackage[section]{placeins}
\usepackage[letterpaper, margin=1in]{geometry}
\usepackage{natbib}
\usepackage[todo,table]{MathEnv}
\usepackage{bbm}
\usepackage[stat,commdiag,alg]{MathShorthand}

\usepackage{breakcites}

\usepackage{tikz}
\usetikzlibrary{3d,calc}
\usepackage{blkarray}
\usepackage{mathtools}

\title{The Global Topology of Orthogonally Decomposable Tensor Landscapes}
\author{Chunyin Siu}
\date{Sep 12, 2026}

\begin{document}

\maketitle

\begin{abstract}
The homogeneous form associated with a symmetric tensor, restricted to the
sphere, is the objective of the best rank-one approximation problem and, with
random coefficients, the energy of a mean-field spin glass. Its critical points have been studied extensively, but the
global organization of the landscape they form is far less understood. We study this
global structure for positive orthogonally decomposable tensors.
We determine the persistent homology of the sublevel and superlevel filtrations in
closed form, for every homological dimension $q$, every ambient dimension $D$ and every tensor order $k$, via a
recurrence in the ambient dimension. Taking the coefficients to be random, we further
prove laws of large numbers for the resulting persistence diagrams at dimension $0$ and $D-2$, giving an
exact description of the typical global topology of a spin-glass-like energy
landscape.
We illustrate our results with numerical computations and simulations.
\end{abstract}

\section{Introduction}

Let $T$ be a symmetric order-$k$ tensor on $\mathbb{R}^D$, and consider its
associated homogeneous degree-$k$ form on the $(D-1)$-dimensional unit sphere $S^{D-1}$. The
critical points of this form are the eigenvectors of $T$, its critical values
the eigenvalues, and its maximizer determines the best rank-one approximation
and the spectral norm of $T$
\citep{lim05_tensor_eigenvalues,qi05_symmetric_tensor_eigenvalues}. This form is
of independent interest to two communities: in data analysis it is the objective
of the best rank-one tensor approximation, the basic primitive of tensor methods
for estimation and signal processing \citep{kolda09_tensor_review}; in
statistical physics, this form with random coefficients is the energy of a mean-field
spin glass, whose local minima are the metastable states of the
system and whose saddles are the barriers between them \citep{auffinger13_spinGlass_criticalPoints}. In either case, while there has been intense interest in the properties of the optima, much less is known about the global organization of these optima in the energy landscape, even though it governs the convergence to these optima and the transition between them.

In this work, we go beyond the study of critical points and study the shape of this global landscape. More precisely, we determine the topology of the
sublevel and superlevel sets of the form, and how it changes across the range
of values.

To this end, we specialize to a tractable subclass of positive orthogonally
decomposable tensors, for which there is an orthonormal basis in which the homogeneous form
is diagonal, i.e.
\begin{equation}\label{eqn:intro_fkD}
f_{k, D}(x) = \frac{1}{k} \sum_{1 \leq i \leq D} \lambda_i x_i^k,
\qquad x \in S^{D-1},
\qquad \lambda_1 \geq \cdots \geq \lambda_D > 0.
\end{equation}
with an appropriate change of basis.
Such tensors, their eigenvectors and singular vectors, and the local behavior
of power iteration on them are well understood
\citep{anandkumar14_tensorDecomposition,robeva16_symmetric_orthogonal_tensor_decomposition,robeva17_orthogonallyDecomposableTensor_singularVector}. Critically for our analysis, its gradient flow decouples across the coordinates, and this enables us to build an inductive argument about the global landscape by exploiting the nestedness of $S^{D-1} \subseteq S^D$.

Despite its relative simplicity, our model case still retains the essential difficulty of the problem: the form is nonconvex, with exponentially many critical points and, for even $k$, exponentially many connected components in its sublevel sets.

We study the global structure of the landscape of $f_{k,D}$ through the
evolution of its sublevel and superlevel sets
$\{x : f_{k,D}(x) \leq t\}$ and $\{x : f_{k,D}(x) \geq t\}$ as the level $t$
varies. The appropriate invariant of such a one-parameter family of spaces is
its persistent homology (defined in \cref{sec:persistent_homology})
\citep{carlsson09_topodata,edelsbrunner10comptopo},
which records in a persistence diagram the level at which each homological
feature, like connected component, loop, and void, is
created and the level at which it is destroyed. 

\subsection{Results and Contributions}

Our work lies at the intersection of nonconvex optimization, numerical
multilinear algebra, and computational topology, and its central contribution
is to compute, in closed form, the entire topological structure of a classical
nonconvex landscape.

\begin{description}
\item [Tensor analysis] Our main results
(\cref{thm:persistence_diagram_sublevel_even,thm:persistence_diagram_sublevel_odd}, and
\cref{cor:persistence_diagram_superlevel_even})
give the persistence diagrams of the sub- and super-level filtrations of
$f_{k,D}$ for every homological dimension, every $D$ and $k$. These results encode how
critical points assemble into a global structure, and in particular, how connected components of local optima merge with one another. This is precisely the structure that governs where descent methods converge.

\item[Computational topology] The field of computational topology is predominantly algorithmic. Given the curse of dimensionality, persistent homology at dimension beyond 2 is rarely computed. Our theoretical computation provides a rare window into high-dimensional topological behavior.

\item[Statistical physics] Finally, we consider the case of random coefficients $\lambda_i$ by proving laws of
large numbers for the resulting persistence diagrams
(\cref{prop:LLN_sublevel_persistence_dim_0_even,prop:LLN_superlevel_persistence_dim_0_even}). As the dimension grows, the diagrams concentrate on explicit deterministic
limits, giving the typical global topology of the landscape across the disorder.
This is, to our knowledge, a first description of the global topology of a
spin-glass-like energy landscape, complementing the extensive study of its critical
points.
\end{description}
Tensor optimization is a notoriously nonconvex problem, and it has been studied
mostly through local and pointwise information: eigenvectors, convergence of
iterations, counts of critical points. Here we use tools from algebraic topology
to describe its global structure, and in the orthogonally decomposable case we
determine that structure completely, laying a foundation for the topological
study of more general landscapes.

\subsection{Discussion}
\label{sec:discussion}

\paragraph{Topological Simplicity and the Importance of Global Topology} Our results show that typically persistence of topological features in $f_{k,D}$ are short, and hence $f_{k,D}$ is in fact topologically not complicated, despite its multitude of critical points.

This simplicity is illustrated in \cref{fig:persistence_digrams_across_dimensions_sublevel_odd,fig:persistence_digrams_across_dimensions_superlevel_odd,fig:persistence_digrams_across_dimensions_sublevel_even,fig:persistence_digrams_across_dimensions_superlevel_even}. From the closed form formulae in our main results, we see that the finite birth and death times only differ only by the contribution of one $\lambda_i$'s among several. We discuss in greater detail in \cref{sec:numerical_observations}.

While readily explainable with our closed form formulae, this topological simplicity is surprising. A priori, $f_{k,D}$ is known to have exponentially many critical points \citep{robeva16_symmetric_orthogonal_tensor_decomposition}, and from here it is tempting to infer that $f_{k,D}$ is very complicated. However, upon inspecting the flow lines that pair up critical points together, we have revealed that these critical points are in fact paired up with critical points with similar critical values, suggesting the nonconvexity resulting from these critical points are in fact mild.

More broadly, the current work highlights the importance of studying the global topology alongside local analysis. A lot of work on tensor analysis and spin glass models measure topological complexity by the count of critical points and distribution of critical values. Our result shows that while these statistics may be complicated (e.g. blowing up exponentially), the resultant topology could be simple (e.g. they pair up to critical points with similar critical values). Therefore, it is important to complement local analysis with the study of global topology.

\paragraph{Future Directions} Our result opens up several future directions.

\begin{description}
\item [General symmetric tensors]
Every symmetric order-$k$ tensor on $\mathbb{R}^D$ admits a symmetric decomposition, so its form can be
written as $g(x) = \tfrac{1}{k}\sum_{i=1}^{r} c_i \langle a_i, x\rangle^k$ for
some coefficients $c_1, ..., c_r$ and vectors $a_1, ..., a_r$. If the $a_i$'s are linearly independent, one may write $g(x) = f_{k,r}(Ax)$, where $A$ is the $r \times D$ matrix, so the relevant variational problem $\max_{x \in S^{D-1}} g(x)$ becomes
$\max_{y \in A(S^{D-1})} f_{k,r}(y)$, the same objective with an ellipsoidal feasible region. As such, understanding the orthogonally decomposable case is the first step towards understanding the general case.

\item[Other homological dimensions in the random model]
Our probabilistic results concern the persistence diagrams at dimensions 0 and $D-2$. For homology in a
fixed finite dimension $q$, or in fixed cofinite dimension $D-q$, the same argument generalizes. The
substantive open case is that of growing dimension, such as $q \sim D/2$, where
the relevant families of subsets grow combinatorially and the limiting behavior
is unclear. We expect their probabilistic behavior would be markedly different from random eigenvalues because of the absence of repulsion.

\item[Spin glass topology] There has been a growing interest in the topology of spin glass models \citep{auffinger13_spinGlass_criticalPoints,aros19_randomTensor}, but little is known about the global topology beyond statistics of critical points. While the topology of general spin glasses is expected to be more complicated than our specialized model, we speculate that the recurrence technique may help attack the topology of general spin glasses.

\end{description}

\subsection{Organization}

The rest of the paper is organized as follows. We review related work in \cref{sec:related_work}. We formally state our results in \cref{sec:main_results}, and present numerical examples in \cref{sec:simulations}. We review preliminaries for our proofs in \cref{sec:preliminaries}. We conduct the foundational Morse-theoretical computation in \cref{sec:morse_computation}. We prove our first main result, \cref{thm:persistence_diagram_sublevel_even}, in \cref{sec:morse_recurrence_even,sec:proof_sublevel_even}. We state Alexander duality in \cref{sec:proof_cor_persistence_diagram_superlevel_even}, and prove \cref{cor:persistence_diagram_superlevel_even} there as well. In \cref{sec:odd_k}, we prove our second main result, \cref{thm:persistence_diagram_sublevel_odd}, using Alexander duality. In \cref{sec:connected_components,sec:probabilistic_results}, we prove the corollaries of our main results. Facts from homological algebra are reviewed in \cref{sec:homological_algebra}.

\subsection*{Acknowledgement}

The author would like to thank Antonio Auffinger for inspiring the question, and Eduardo Paluzo-Hidalgo, Brice Huang, and Benjamin Thompson for insightful discussions. The author was supported by the Croucher Foundation.

Claude Opus 5.0 was used to draft the introduction, the literature review, polish the language, aid numerical computations, and Harmonic's Aristotle and Claude Opus 5.0 were used to fix typographical mistakes, and to conduct the verification that $f_{k, D}$ satisfies the Smale condition. The main results were proven \emph{without} the use of large language models.

\section{Related Work}
\label{sec:related_work}

The present work is the confluence of three lines of research.

\begin{description}
\item [Tensor analysis] The
variational theory of tensor eigenvalues and singular values was introduced by
\citep{lim05_tensor_eigenvalues} and \citep{qi05_symmetric_tensor_eigenvalues};
see \citep{kolda09_tensor_review} for a survey.
Orthogonally decomposable tensors, whose associated form is
diagonal~\eqref{eqn:intro_fkD}, are characterized in \citep{robeva16_symmetric_orthogonal_tensor_decomposition},
with their singular vectors analyzed in \citep{robeva17_orthogonallyDecomposableTensor_singularVector}; this structure
also underlies guaranteed tensor-decomposition methods for latent-variable
models \citep{anandkumar14_tensorDecomposition}.
These works describe the eigenvectors, singular vectors, and local optimization
dynamics; we complement them by computing the global topology of the landscape.
\item [Morse theory and persistent homology]
The idea of understanding a space through the changing topology of the sublevel
sets of a function on it originates in Morse theory
\citep{milnor63_morseTheory}.
Persistent homology, popularized by the rise of Topological Data Analysis (TDA), makes this idea quantitative and multiscale, recording the
births and deaths of homological features across all levels in a persistence
diagram \citep{carlsson09_topodata, edelsbrunner10comptopo}. The field of TDA, however, is predominantly computational, and persistence diagrams are
typically obtained algorithmically from raw data. We instead give the
diagrams, in closed form, of an important class of functions in numerical multilinear algebra and optimization theory.
\item[Statistical physics, random field theory, and random topology] For random tensors, the statistics of critical points, in particular their expected number via Kac-Rice formulae, have been
studied extensively in connection with spin glasses and the spiked-tensor model
\citep{auffinger13_spinGlass_criticalPoints,aros19_randomTensor}. The global topology of the landscape has received less
attention, and our probabilistic results address this gap. A separate
line of work studies the topology of random spaces, e.g.
simplicial complexes \citep{kahle09_randomCliqueComplex,hiraoka18_limiting_persistenceDiagram,bobrowski18_randomGeometricComplexes_survey,owada20_bettiNumber_processLevel,siu25_preferentialAttachment_homology} and percolation \citep{bobrowski20_percolation,duncan23_topologicalPercolation_torus}. Such spaces are often combinatorial in nature (\citep{perez23_persistentHomology_stochasticProcess} is a notable exception). Ours is a rare continuous and exactly solvable instance. Closely related to our work is the literature on random field theory, which studies smooth random
fields $f$ and their \emph{exceedance probability} $\mathbb{P}(\sup f \geq u)$
for large $u$. This probability is approximated, via the Gaussian kinematic formula, by the expectation of a topological quantity, namely the Euler characteristic of the superlevel sets (also known as excursion sets) \citep{adler81_randomFieldTheory,adler09_randomFieldTheory}. The Euler characteristic is a single aggregate invariant, examined in the extreme regime; we instead
exactly recover the full persistent homology, in every dimension and across all levels.
\end{description}

\section{Main Results}
\label{sec:main_results}

Let $k$ and $D$ be integers such that $k > 2$ and $D \geq 1$. Let $$\lambda_1 \geq \lambda_2 \geq ... \geq \lambda_D > 0.$$
Denote the $(D-1)$-dimensional unit sphere by
$$S^{D-1} = \{(x_1, ..., x_D) \in \mathbb{R}^D: \sum_{1 \leq i \leq D} x_i^2 = 1\}.$$
Consider $f_{k, D}: S^{D-1} \to \mathbb{R}$ defined by
$$f_{k, D}(x) = \frac{1}{k} \sum_{1 \leq i \leq D} \lambda_i x_i^k.$$
In this paper, we investigate the topology of $f_{k, D}$. We first set up the following notation.
\begin{align}
\nu_E &= \frac{1}{k} \left(\sum_{i \in E} \lambda_i^{-\frac{2}{k-2}} \right)^{-\frac{k-2}{2}} = \frac{1}{k} \|\lambda\|_{L^{-\frac{2}{k-2}}(E)} & \text{ for } E \subseteq \{1, ..., D\}, \label{eqn:nu_general_set}\\
\nu_\xi &= \nu_{\{i: \xi_i \neq 0\}} & \text{ for } \xi \in \{-1, 0, 1\}^D. \label{eqn:nu_signum_vector}
\end{align}
Note that the norm notation $\|\cdot\|$ in \cref{eqn:nu_general_set} is an abuse of notation, as $-\frac{2}{k-2} \in [-2, 0)$ for $k \geq 3$.

\begin{table}
\centering
\caption{Points in persistence diagrams of the persistent homology $H_q(X_t; \mathbb{Z}/2\mathbb{Z})$ of $f_{k,D}$'s sublevel and superlevel filtrations $X_t$ with coefficients in $\mathbb{Z}/2\mathbb{Z}$, where $D \geq 1$. See \cref{sec:persistent_homology,sec:homological_algebra} for definitions.}
\label{tab:persistence_diagrams_fkD}
\newcolumntype{C}[1]{>{\centering\arraybackslash}p{#1}}
\begin{tabular}{|C{0.8cm}|C{7.3cm}|C{1.7cm}||C{3cm}|C{1.1cm}|}
\hline
dim ($q$) & \centering points with finite death times & multi-plicities & points with infinite death times & multi-plicities
\\ \hline \hline
\multicolumn{5}{|c|}{Even $k$, sublevel filtration $X_t = \{x: f_{k,D}(x) \leq t \}$. See \cref{thm:persistence_diagram_sublevel_even}.} \\
\hline
$0$ & $(\nu_{\{1, ..., D\}}, \nu_{\{1, ..., D\} - \{i\}})$ for $1 \leq i \leq D$ & $2^{D-i}$ & $(\nu_{\{1, ..., D\}}, \infty)$ & 1
\\ \hline
$1$ to $D-3$ & $(\nu_E, \nu_{E - \{i\}})$, where $E \subseteq \{1, ..., D\}$, $|E| = D-q$, $\{1, ..., i\} \subseteq E$ & $2^{|E| - i}$ & none & N/A
\\ \hline
$D-2$ & $(\nu_{\{1, i\}}, \nu_{\{i\}})$ for $2 \leq i \leq D$; $(\nu_{\{1, 2\}}, \nu_{\{1\}})$ & 2; 1 (resp.) & none & N/A
\\ \hline
$D-1$ & none & N/A & $(\nu_{\{1\}}, \infty)$ & 1
\\ \hline
$\geq D$ & none & N/A & none & N/A
\\ \hline \hline
\multicolumn{5}{|c|}{Even $k$, superlevel filtration $X_t = \{x: f_{k,D}(x) \geq t \}$. See \cref{cor:persistence_diagram_superlevel_even}.} \\
\hline
$0$ & $(\nu_{\{1, i\}}, \nu_{\{i\}})$ for $2 \leq i \leq D$; $(\nu_{\{1, 2\}}, \nu_{\{1\}})$ & 2; 1 (resp.) & $(-\infty, \nu_{\{1\}})$ & 1
\\ \hline
$1$ to $D-3$ & $(\nu_E, \nu_{E - \{i\}})$, where $E \subseteq \{1, ..., D\}$, $|E| = q+2$, $\{1, ..., i\} \subseteq E$ & $2^{|E| - i}$ & none & N/A
\\ \hline
$D-2$ & $(\nu_{\{1, ..., D\}}, \nu_{\{1, ..., D\} - \{i\}})$ for $1 \leq i \leq D$ & $2^{D-i}$ & none & N/A
\\ \hline
$D-1$ & none & N/A & $(-\infty, \nu_{\{1, ..., D\}})$ & 1
\\ \hline
$\geq D$ & none & N/A & none & N/A
\\ \hline \hline
\multicolumn{5}{|c|}{Odd $k$, sublevel filtration $X_t = \{x: f_{k,D}(x) \leq t \}$. See \cref{thm:persistence_diagram_sublevel_odd}.} \\
\multicolumn{5}{|c|}{(The points for the superlevel filtration with odd $k$ are the same except that}\\
\multicolumn{5}{|c|}
{signs are flipped and birth and death are flipped.)}\\
\hline
$0$ & $(-\nu_{\{i\}}, -\nu_{\{1, i\}})$ for $2 \leq i \leq D$; $(\nu_{\{1, ..., D\}}, \nu_{\{2, ..., D\}})$ & 1; 1 & $(-\nu_{\{1\}}, \infty)$ & 1
\\ \hline
$1$ to $D-3$ & $(-\nu_{E^-}, -\nu_{{E^-} \cup \{1\}}), (\nu_{{E^+} \cup \{1\}}, \nu_{E^+})$, where $E^\pm \subseteq \{2, ..., D\}$, $|E^-| = q+1$, $|E^+| = D-q-1$ & 1, 1 & none & N/A
\\ \hline
$D-2$ & \parbox{7.5cm}{\centering $(-\nu_{\{2, ..., D\}}, -\nu_{\{1, ..., D\}})$; \\ \centering $(\nu_{\{1, i\}}, \nu_{\{i\}})$ for $2 \leq i \leq D$} & 1; 1 & none & N/A
\\ \hline
$D-1$ & none & N/A & $(\nu_{\{1\}}, \infty)$ & 1
\\ \hline
$\geq D$ & none & N/A & none & N/A
\\ \hline
\end{tabular}
\end{table}

Our first main result is as follows.

\begin{theorem}\label{thm:persistence_diagram_sublevel_even}
Consider the persistence diagrams of the sublevel filtration (defined in \cref{sec:persistent_homology}) of $f_{k, D}: S^{D-1} \to \mathbb{R}$ with homological coefficients in $\mathbb{Z}/2\mathbb{Z}$ (defined in \cref{def:singular_homology}). Suppose $k$ is even.
\begin{enumerate}
\item For points with infinite death times, their birth times are as follows:
\begin{align*}
& \nu_{\{1, ..., D\}} & \text{ at dimension } 0
\\
& \nu_{\{1\}} & \text{ at dimension } D-1,
\end{align*}
and their multiplicities are both $1$.
\item The points with finite death times at dimension $q$, where $0 \leq q \leq D-2$, are precisely points of the homogeneous form
$$(\nu_E, \nu_{E - \{i\}}),$$
where $E$ ranges over subsets of $\{1, ..., D\}$ with $D - q$ elements such that $1 \in E$, whereas $i$ ranges over elements of $E$ such that $\{1, ..., i\} \subseteq E$. Further, the multiplicity of $(\nu_E, \nu_{E - \{i\}})$ is $2^{|E| - i}$.
\item The persistence diagram at dimension $D - 1$ or higher has no point with finite death time.
\end{enumerate}
\end{theorem}

Note that $\nu_{\{1\}} = \lambda_1/k$. More generally, $\nu_{\{i\}} = \lambda_i/k$. The above result is summarized in the top block of \cref{tab:persistence_diagrams_fkD}, and its proof spans \cref{sec:morse_computation,sec:morse_recurrence_even,sec:proof_sublevel_even}.

Alexander duality (\cref{thm:alexander_duality}) gives the following corollary, which we prove in detail in \cref{sec:proof_cor_persistence_diagram_superlevel_even}. Note that in a superlevel set filtration, we call the larger coordinate the birth time to be consistent with the direction of the filtration. We maintain the convention that the first coordinate is less than or equal to the second coordinate.

\begin{corollary}\label{cor:persistence_diagram_superlevel_even}
Consider the persistence diagrams of the superlevel filtration of $f_{k, D}: S^{D-1} \to \mathbb{R}$ with homological coefficients in $\mathbb{Z}/2\mathbb{Z}$. Suppose $k$ is even.
\begin{enumerate}
\item For points with negative infinite death times, their birth times are as follows:
\begin{align*}
& \nu_{\{1, ..., D\}} & \text{ at dimension } D-1
\\
& \nu_{\{1\}} & \text{ at dimension } 0,
\end{align*}
and their multiplicities are both $1$.
\item The points with finite death times at dimension $q$, where $0 \leq q \leq D-2$, are precisely points with finite death time in the persistence diagram of the sublevel filtration at dimension $D-q-2$ of the form, i.e. points of the form
$$(\nu_E, \nu_{E - \{i\}}),$$
where $E$ ranges over subsets of $\{1, ..., D\}$ with $2 + q$ elements such that $1 \in E$, whereas $i$ ranges over elements of $E$ such that $\{1, ..., i\} \subseteq E$. Further, the multiplicity of $(\nu_E, \nu_{E - \{i\}})$ is $2^{|E| - i}$.
\item The persistence diagram at dimension $D - 1$ or higher has no point with finite death time.
\end{enumerate}
\end{corollary}

The above corollary is summarized in the middle block of \cref{tab:persistence_diagrams_fkD}.

Our second main result is as follows.

\begin{theorem}\label{thm:persistence_diagram_sublevel_odd}
Consider the persistence diagrams of the sublevel filtration (defined in \cref{sec:persistent_homology}) of $f_{k, D}: S^{D-1} \to \mathbb{R}$ with homological coefficients in $\mathbb{Z}/2\mathbb{Z}$ (defined in \cref{def:singular_homology}). Suppose $k$ is odd.
\begin{enumerate}
\item For points with infinite death times, their birth times are as follows:
\begin{align*}
& -\nu_{\{1\}} & \text{ at dimension } 0
\\
& \nu_{\{1\}} & \text{ at dimension } D-1,
\end{align*}
and their multiplicities are both $1$.
\item The points with finite death times at dimension $q$, where $0 \leq q \leq D-2$, are precisely points of the form
$$(-\nu_{E^-}, -\nu_{E^- \cup \{1\}}), (\nu_{E^+ \cup \{1\}}, \nu_{E^+})$$
where $E^-$ and $E^+$ range over subsets of $\{2, ..., D\}$ (note that $1 \notin E^-, E^+$) such that $E^-$ has exactly $q + 1$ elements and $E^+$ has exactly $D - q - 1$ elements. Their multiplicities are all $1$.
\item The persistence diagram at dimension $D - 1$ or higher has no point with finite death time.
\end{enumerate}
\end{theorem}

By the symmetry of $f_{k,D}$ when $k$ is odd, the superlevel filtration has the same persistence diagrams. The above result is summarized in the bottom block of \cref{tab:persistence_diagrams_fkD}, and is proven in \cref{sec:odd_k}.

\subsection{Corollaries of Main Results}

For concreteness, we describe the connected components of the sublevel sets and superlevel sets of $f_{k, D}$ in \cref{prop:components_sublevel_even,prop:components_superlevel_even,prop:components_sublevel_odd}. These propositions will be proven in \cref{sec:connected_components}.

We first set up some notation. Fix $D$ and let 
\begin{equation}\label{eqn:Xi_set_of_signum_vectors}
\Xi = \{-1, 0, 1\}^D - \{(0, ..., 0)\}.
\end{equation}
We call elements of $\Xi$ signum vectors.
For each signum vector $\xi \in \Xi$, let
$$H^\xi = \{x \in S^{D-1} : x_i > 0 \text{ whenever } \xi_i > 0, x_i < 0 \text{ whenever } \xi_i < 0\}.$$

\begin{proposition}\label{prop:components_sublevel_even}
Let $i \in \{1, ..., D\}$. Suppose $\nu_{\{1, ..., D\} - \{i - 1\}} \leq t < \nu_{\{1, ..., D\} - \{i\}}$. Let $\Xi_i = \{\xi \in \Xi: \xi_1 = ... = \xi_{i-1} = 0, 0 \notin \{\xi_i, ..., \xi_D\}\}$.
Then each connected component of $\{x \in S^{D-1}: f_{k, D} \leq t\}$ lies in $H^\xi$ for some $\xi \in \Xi_i$, and each $H^\xi$ with $\xi \in \Xi_i$ contains exactly one component.
\end{proposition}

\begin{proposition}\label{prop:components_superlevel_even}
For even $k$, the local maxima of $f_{k, D}$ are $e_1, ..., e_D, -e_1, ..., -e_D$, with $e_1, -e_1$ being the global maxima. In the superlevel filtration of $f_{k, D}$, the connected components of $e_i$ and $-e_i$ merge with that of $e_1$ at the level $\nu_{\{1, i\}}$. The connected component of $-e_1$ merges with that of $e_1$ at the level $\nu_{\{1, 2\}} \geq \nu_{\{1, j\}}$ for $j \geq 2$.
\end{proposition}

\begin{proposition}\label{prop:components_sublevel_odd}
For odd $k$, the local minima of $f_{k, D}$ are $-e_1, ..., -e_D$ and a point $x^{\mathbf{1}_D} = x^{\sum_{1 \leq i \leq D} e_i}$, precisely defined in \cref{eqn:critical_points}, whose coordinates are all strictly positive. The global minimum is $-e_1$. In the sublevel filtration of $f_{k,D}$, the connected component of $-e_i$ merges with that of $-e_1$ at the level $-\nu_{\{1, i\}}$. At the level of $f_{k,D}(x^{\mathbf{1}_D})$, all the connected components of the $-e_i$'s have already merged, and the connected component of $x^{\mathbf{1}_D}$ merges with this component at the level $\nu_{\{2, ..., D\}}$.
\end{proposition} 

Now, we consider random $f_{k, D}$, where the $\lambda_i$'s are random. Random tensors arise naturally in real data and in statistical physics \citep{aros19_randomTensor,bi21_tensor_in_statistics}. Specifically, we consider the random model where $\lambda_i$'s are eigenvalues of an oversampled / critical Wishart ensemble.

\begin{definition}[Wishart ensemble]\label{def:wishart}
Let $X$ be a $D \times n$ random matrix with independent normally distributed entries $X_{ij}$ with mean $0$ and variance $1$. The Wishart ensemble is defined as the random matrix
\begin{equation*}
W = \frac{1}{n} XX^T.
\end{equation*}
The aspect ratio of the ensemble is defined as $\gamma = n/D$. The ensemble is said to be oversampled, critical, and undersampled if, respectively, $\gamma > 1$, $\gamma = 1$ and $\gamma < 1$.
\end{definition}

To describe the limiting behavior of the Wishart ensemble, we recall the notion of weak convergence.
\begin{definition}\label{def:weak_convergence}
Recall that a sequence of measures $\mu_n$ on a topological space is said to converge weakly to a measure $\mu$ if $\int \varphi d\mu_n \to \int \varphi d\mu$ for every bounded continuous function $\varphi$.
\end{definition}

It is well known that the eigenvalues $\lambda_1 \geq ... \geq \lambda_D > 0$ of an oversampled Wishart ensemble have a limiting distribution (Theorem 3.6 of \citep{bai2010_randomMatrixTheory}), in the sense that the empirical measure $(1/D)\sum_{1 \leq  i \leq D} \delta_{\lambda_i}$ converges weakly to the Marchenko-Pastur distribution, whose density is
\begin{equation}
\rho_{\text{MP}}(\lambda; \gamma) = \frac{\gamma}{2\pi\lambda} \sqrt{(\lambda_+ - \lambda)(\lambda - \lambda_-)}, \quad \lambda \in [\lambda_-, \lambda_+],
\end{equation}
where $\lambda_\pm = (1 \pm 1/\sqrt{\gamma})^2$.

We consider the persistence diagram of the sublevel filtration at dimension 0. Since all birth times are the same ($\nu_{\{1, ..., D\}}$), we consider only the persistence (death time minus birth time, see \cref{def:persistence_diagram}).

\begin{proposition}[Law of Large Numbers for the Sublevel Persistence at Dimension 0 for Even $k$]\label{prop:LLN_sublevel_persistence_dim_0_even}
Let $k \geq 4$ be an even integer. Let $D \geq 2$ be an integer. Assume $\lambda_1, ..., \lambda_D$ are eigenvalues of an oversampled Wishart Ensemble with aspect ratio $\gamma > 1$. Denote by $\mathcal{D}_0^\text{finite}$ the persistence diagram of the sublevel filtration of $f_{k, D}$ (with $\lambda_i$'s as random eigenvalues as assumed) at dimension 0, with points with infinite endpoints dropped. Consider the scaled empirical measure of the persistence of points in this diagram
$$\mathcal{L}_D = \frac{1}{2^D} \sum_{(b, d) \in \mathcal{D}_0^\text{finite}} \delta_{D^{k/2}(d - b)}.$$
Then $\mathcal{L}_D$ weakly converges to the Dirac delta measure at $\frac{k-2}{2k} I_k^{-\frac{k}{2}} \lambda_+^{-\frac{2}{k-2}}$, where
$I_k = \int_{\lambda_-}^{\lambda_+} \lambda^{-\frac{2}{k-2}}\rho_{\text{MP}}(\lambda) d\lambda$.
\end{proposition}

This proposition is proven in \cref{sec:probabilistic_results}.

\begin{proposition}[Law of Large Numbers for the Superlevel Persistence Diagram at Dimension 0 for Even $k$]\label{prop:LLN_superlevel_persistence_dim_0_even}
Let $k \geq 4$ be an even integer. Let $D \geq 2$ be an integer. Assume $\lambda_1, ..., \lambda_D$ are eigenvalues of an oversampled Wishart Ensemble with aspect ratio $\gamma > 1$. Denote by $\mathcal{D}_0^\text{finite}$ the persistence diagram of the superlevel filtration of $f_{k, D}$ (with $\lambda_i$'s as random eigenvalues as assumed) at dimension 0, with points with infinite coordinates dropped. Consider the scaled empirical measure of the points in this diagram
$$\mathcal{L}_D = \frac{1}{D} \sum_{(d, b) \in \mathcal{D}_0^\text{finite}} \delta_{(d, b)}.$$
Note that $\mathcal{L}_D$ is a measure on $\mathbb{R}^2$.
Denote by $\nu^* \rho_{MP}(\lambda) d\lambda$ the pushforward measure of $\rho_{MP}(\lambda) d\lambda$ on the curve $\lambda \mapsto \left(\frac{1}{k} \left(\lambda^{-\frac{2}{k-2}} + \lambda_+^{-\frac{2}{k-2}}\right)^{-\frac{k-2}{2}}, \lambda/k \right)$, where $\lambda \in [\lambda_-, \lambda_+]$. Note that $\nu^* \rho_{MP}(\lambda) d\lambda$ is singular with respect to the Lebesgue measure on $\mathbb{R}^2$.
Then $\mathcal{L}_D$ weakly converges to $2 \nu^* \rho_{MP}(\lambda) d\lambda$.
\end{proposition}

This proposition is also proven in \cref{sec:probabilistic_results}.

\begin{remark}
Observing the similarity between the first rows of the middle and bottom blocks of \cref{tab:persistence_diagrams_fkD}, the Laws of Large Numbers for the super- and sub-level persistence diagrams at dimension 0 for \emph{odd} $k$ are the same as the one above in \cref{prop:LLN_superlevel_persistence_dim_0_even}, except that the weak limits are $\nu_* \rho_{MP}(\lambda) d\lambda$ and $(-\nu)_* \rho_{MP}(\lambda) d\lambda$ respectively, where $-\nu$ denotes the transformation of $\lambda$ in \cref{prop:LLN_superlevel_persistence_dim_0_even} post-composed by $(x, y) \mapsto (-y, -x)$.
\end{remark}

\section{Numerical Computations and Simulations}
\label{sec:simulations}

\subsection{Setup and Figure Layout}
We illustrate our main results with numerical computations and simulations. For deterministic persistent diagrams \cref{thm:persistence_diagram_sublevel_even,thm:persistence_diagram_sublevel_odd}, and
\cref{cor:persistence_diagram_superlevel_even}, we consider the cases when the ambient dimension is $D = 9$ and the tensor orders are $k = 3$ and $k = 4$. We consider several choices of $\lambda_i$'s:
\begin{description}
\item[Exponential] $\lambda_i = 0.9^i$
\item[Linear] $\lambda_i = 1 - i/D$
\item[Wishart] $\lambda_i$'s are eigenvalues of an instance of the Wishart ensemble (the random matrix defined in \cref{def:wishart}) with $\gamma = 2$
\item[Cliff]$\lambda_1 = 1$, $\lambda_2 = 0.99$, $\lambda_3 = 0.98$ (here comes the cliff), $\lambda_4 = 0.1$, $\lambda_5 = 0.19$, $\lambda_6 = 0.18$, ..., $\lambda_9 = 0.15$.
\end{description}
The exponential and linear cases illustrate the situation with smoothly decaying coefficients. The third choice illustrates a the behavior of random coefficients. The cliff case illustrates the situation when a few (in this case, the first three) coefficients dominate.

For $k = 3$, their persistence diagrams are shown in \cref{fig:persistence_digrams_across_dimensions_sublevel_odd} (sublevel) and \cref{fig:persistence_digrams_across_dimensions_superlevel_odd} (superlevel). For $k = 4$, their persistence diagrams are shown in \cref{fig:persistence_digrams_across_dimensions_sublevel_even} (sublevel) and \cref{fig:persistence_digrams_across_dimensions_superlevel_even} (superlevel).

\subsection{Observations}
\label{sec:numerical_observations}

\paragraph{Duality} From these figure, we observe duality manifesting in two ways.

\begin{description}
\item[Duality between sub- and super-level filtration] The superlevel persistence diagrams and sublevel persistence diagrams are mirror images of each other. For instance, the blue dots (homological dimension $q = 0$) in the first column of \cref{fig:persistence_digrams_across_dimensions_sublevel_odd} are the same as the orange dots (homological dimension $q = 8$) in the last column of \cref{fig:persistence_digrams_across_dimensions_superlevel_odd}. This is a simple corollary of Alexander duality (\cref{thm:alexander_duality}).
\item[Duality for odd tensor order] In \cref{fig:persistence_digrams_across_dimensions_sublevel_odd}, the persistent diagrams are mirror images across dimensions as well. For instance The blue dots (homological dimension $q = 0$) in the first column is the same as the orange dots (homological dimension $q = 8$) in the last column with a sign flip. This is due to the oddity of $f_{k,D}$ for odd $k$ and Alexander duality (\cref{thm:alexander_duality}).
\end{description}

\paragraph{Short persistence} We also observe that most points lie very close to the diagonal (the line where birth = death), especially for interior dimensions (far from both 0 and $D-1$). This is because each point in the persistence diagrams are of the form $(\nu_E, \nu_{E - \{i\}})$ (or reversed) for some set $E$ and some $i \in E$, so the persistence is smaller whenever $i$'s contribution to $\nu_E$ is not big.

\begin{figure}[t]
\centering
\includegraphics[width = 0.9\linewidth]{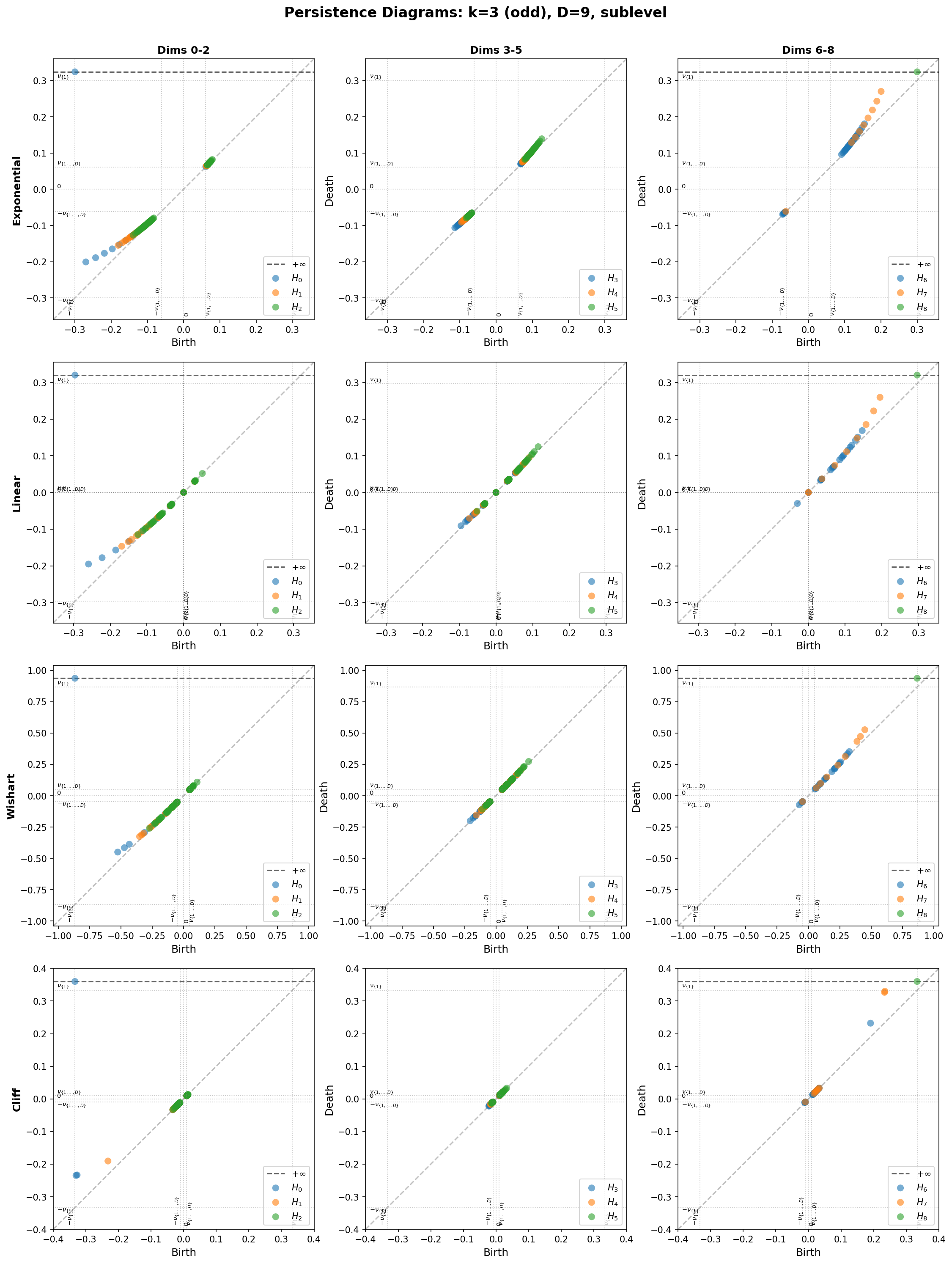}
\caption{Persistence diagrams of the persistent homology $H_q(X_t; \mathbb{Z}/2\mathbb{Z})$ of $f_{k,D}$'s, with $k = \mathbf{3}$, \textbf{sublevel} filtrations $X_t$ with coefficients in $\mathbb{Z}/2\mathbb{Z}$, where $D = 9$ and the coefficients $\lambda_i$'s are specified in \cref{sec:simulations}. The first columns correspond to the persistence diagrams at homological dimensions $0-2$, the other columns, $3-5$ and $6-8$. See \cref{sec:persistent_homology,sec:homological_algebra} for definitions.}
\label{fig:persistence_digrams_across_dimensions_sublevel_odd}
\end{figure}

\begin{figure}[t]
\centering
\includegraphics[width = 0.9\linewidth]{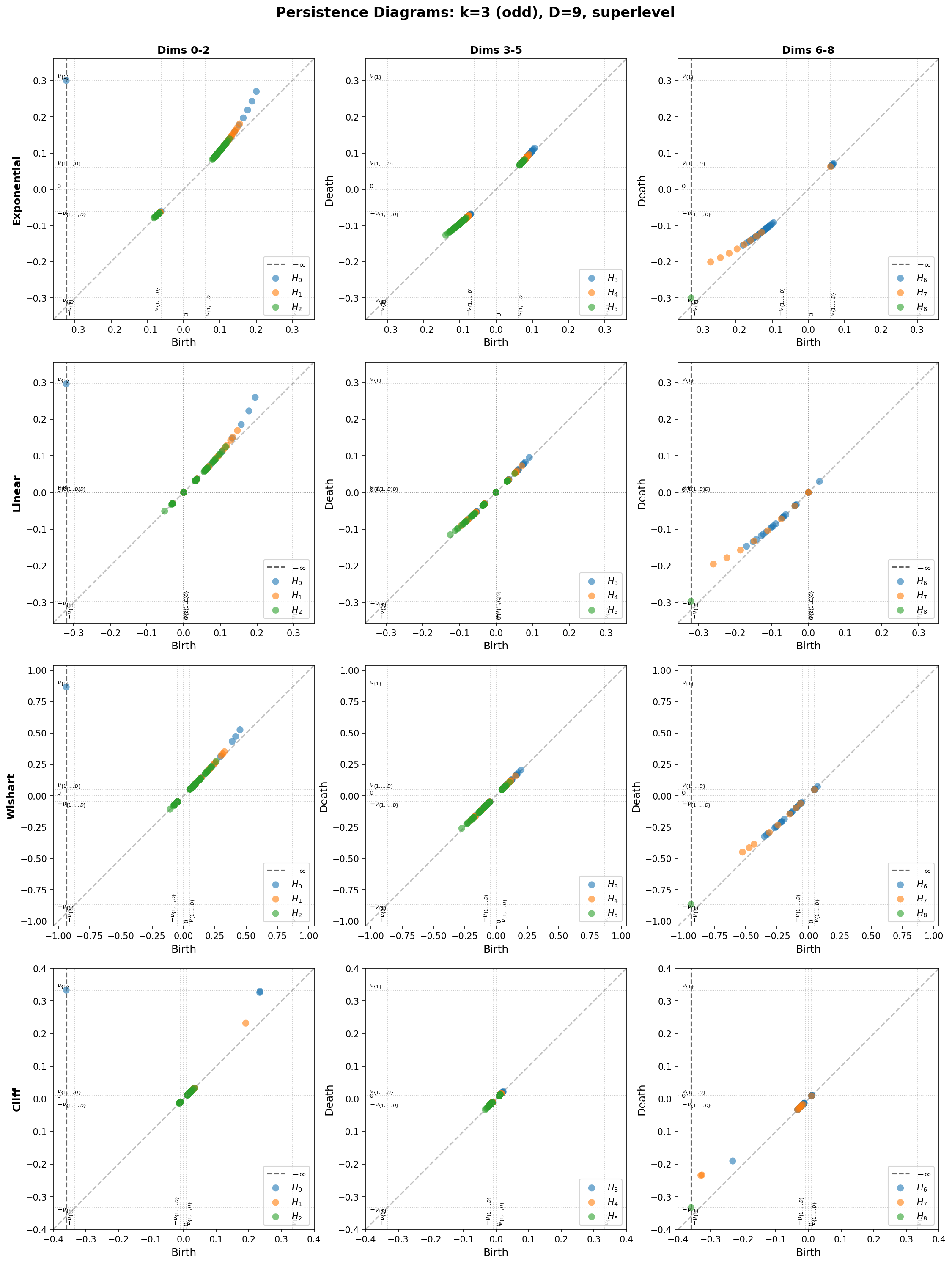}
\caption{Persistence diagrams of the persistent homology $H_q(X_t; \mathbb{Z}/2\mathbb{Z})$ of $f_{k,D}$'s, with $k = \mathbf{3}$, \textbf{superlevel} filtrations $X_t$ with coefficients in $\mathbb{Z}/2\mathbb{Z}$, where $D = 9$ and the coefficients $\lambda_i$'s are specified in \cref{sec:simulations}. The first columns correspond to the persistence diagrams at homological dimensions $0-2$, the other columns, $3-5$ and $6-8$. See \cref{sec:persistent_homology,sec:homological_algebra} for definitions.}
\label{fig:persistence_digrams_across_dimensions_superlevel_odd}
\end{figure}

\begin{figure}[t]
\centering
\includegraphics[width = 0.9\linewidth]{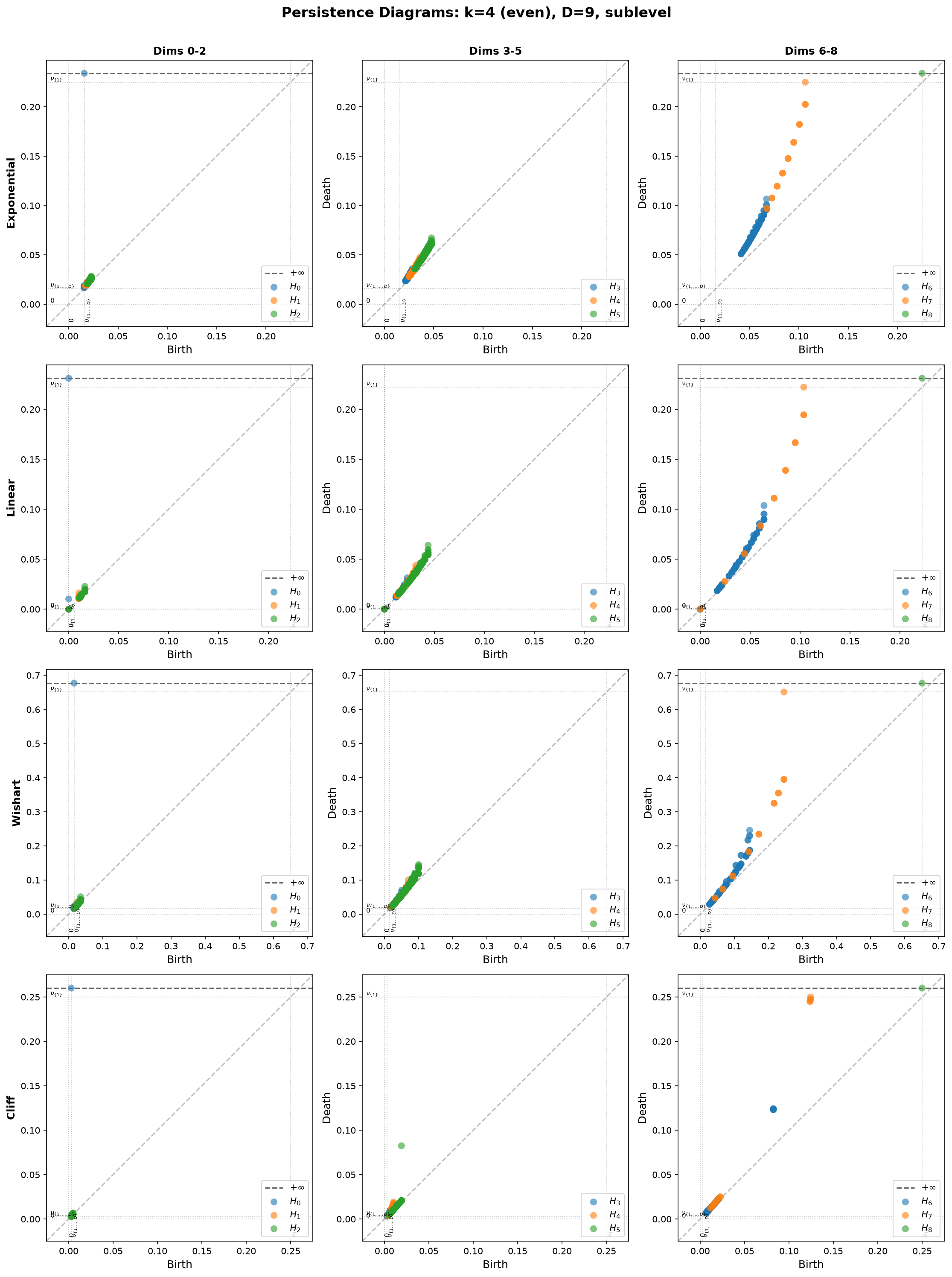}
\caption{Persistence diagrams of the persistent homology $H_q(X_t; \mathbb{Z}/2\mathbb{Z})$ of $f_{k,D}$'s, with $k = \mathbf{4}$, \textbf{sublevel} filtrations $X_t$ with coefficients in $\mathbb{Z}/2\mathbb{Z}$, where $D = 9$ and the coefficients $\lambda_i$'s are specified in \cref{sec:simulations}. The first columns correspond to the persistence diagrams at homological dimensions $0-2$, the other columns, $3-5$ and $6-8$. See \cref{sec:persistent_homology,sec:homological_algebra} for definitions.}
\label{fig:persistence_digrams_across_dimensions_sublevel_even}
\end{figure}

\begin{figure}[t]
\centering
\includegraphics[width = 0.9\linewidth]{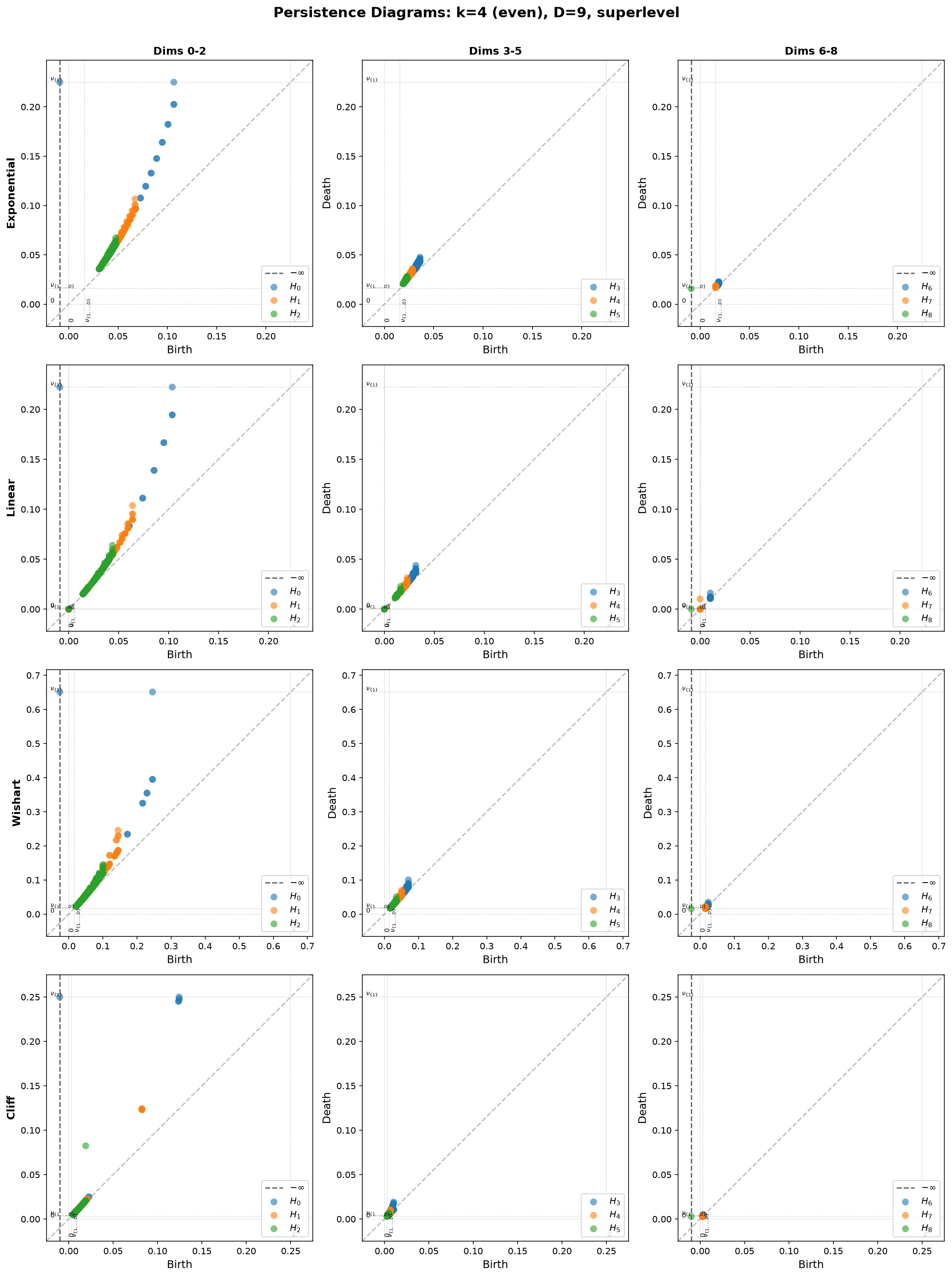}
\caption{Persistence diagrams of the persistent homology $H_q(X_t; \mathbb{Z}/2\mathbb{Z})$ of $f_{k,D}$'s, with $k = \mathbf{4}$, \textbf{superlevel} filtrations $X_t$ with coefficients in $\mathbb{Z}/2\mathbb{Z}$, where $D = 9$ and the coefficients $\lambda_i$'s are specified in \cref{sec:simulations}. The first columns correspond to the persistence diagrams at homological dimensions $0-2$, the other columns, $3-5$ and $6-8$. See \cref{sec:persistent_homology,sec:homological_algebra} for definitions.}
\label{fig:persistence_digrams_across_dimensions_superlevel_even}
\end{figure}

\subsection{Simulations for Probabilistic Results}

To illustrate our probabilistic result for the sublevel filtration \cref{prop:LLN_sublevel_persistence_dim_0_even}, we consider the case where the tensor order is $k = 4$, the coefficients $\lambda_i$'s are Wishart eigenvalues with $\gamma = 2$, and the homological dimension is $q = 0$. On the left panel of \cref{fig:persistence_digrams_random}, we show the cumulative distributions of the scaled persistences and the limiting distribution for the sublevel filtration with the ambient dimensions $D$ increasing from 10 to 160. On the right panel, we show the empirical persistence diagram and its limiting curve with ambient dimension $D = 80$.

\begin{figure}[h]
\centering
\includegraphics[width = 0.45\linewidth]{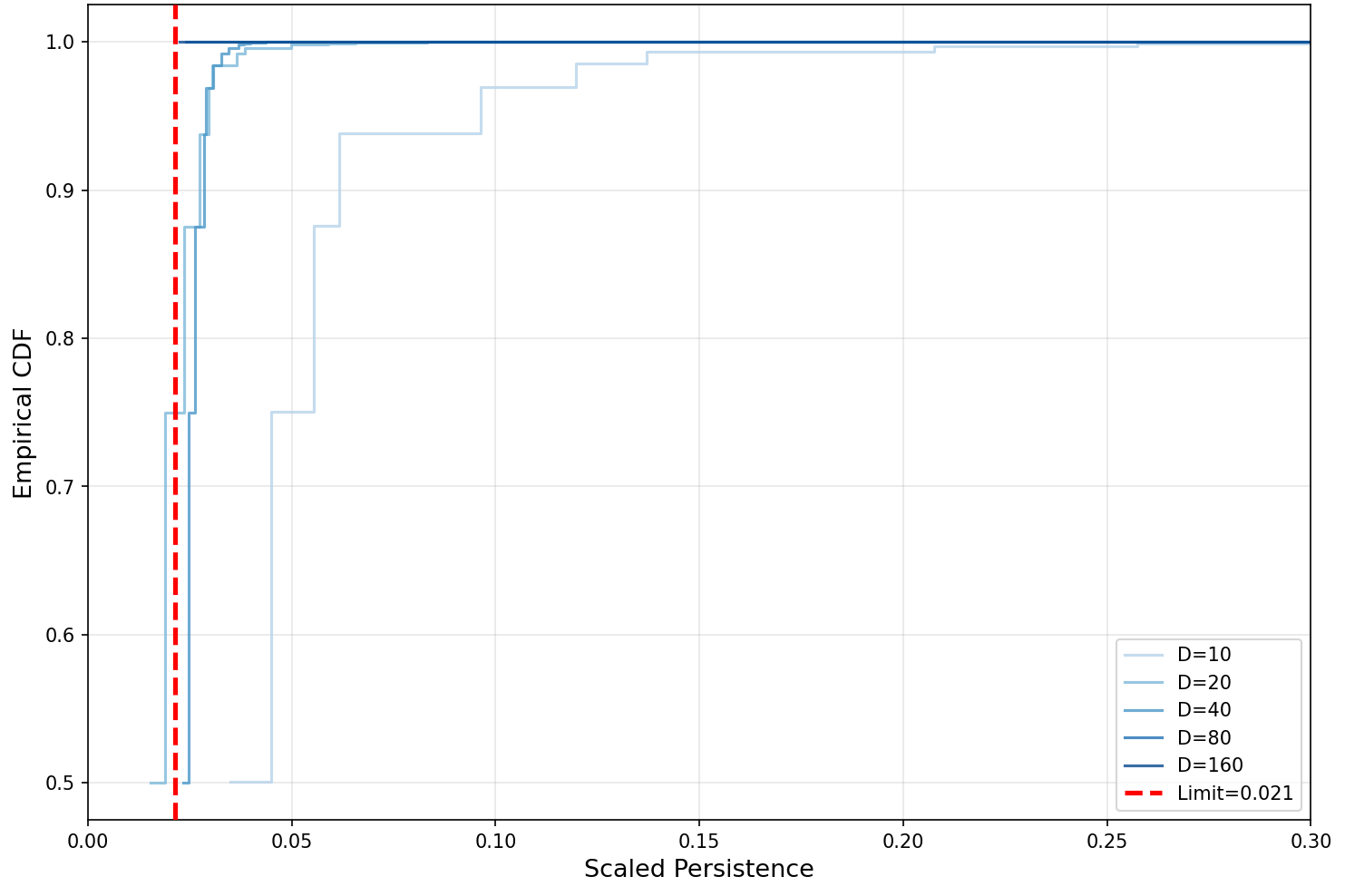}
\includegraphics[width = 0.45\linewidth]{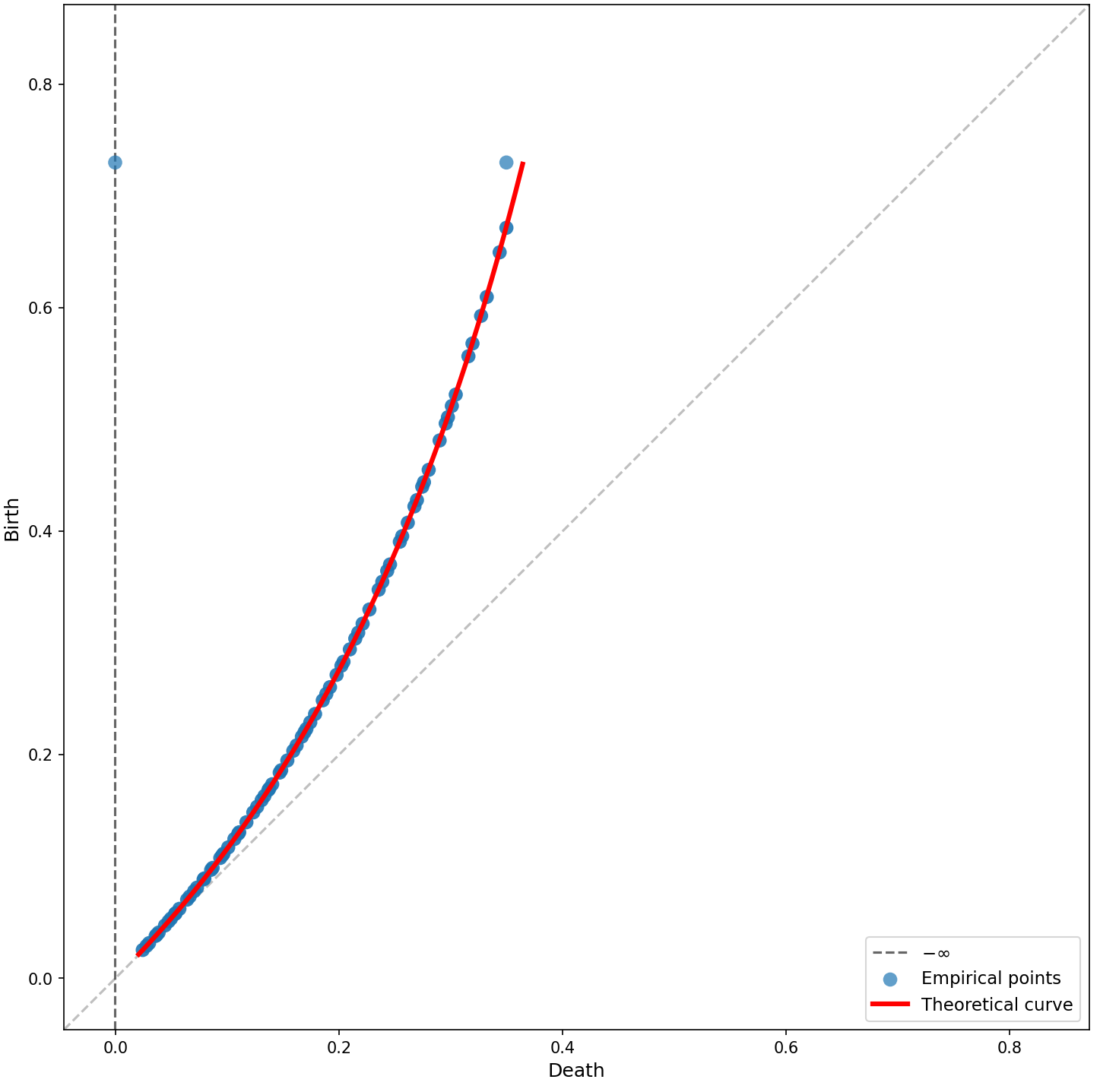}
\caption{Convergence of persistent homology of $f_{k,D}$'s, where $k = \mathbf{4}$ and the coefficients $\lambda_i$'s are Wishart eigenvalues with $\gamma = 2$. On the left panel, we show the convergence of the cumulative distribution functions of scaled persistence as the ambient dimension $D$ increases from $10$ to $160$. On the left, we compare the persistence diagram with ambient dimension $D = 80$ with the theoretical limit.}
\label{fig:persistence_digrams_random}
\end{figure}

\section{Preliminaries}
\label{sec:preliminaries}

In this section, we review major definitions and results from persistent homology, Morse theory, and computational topology to make precise our main results and to prepare for their proofs. We defer a review on homological algebra and algebraic topology (e.g. chain complex, commutative diagram, homology, cohomology, and reduced (co-) homology) to the \cref{sec:homological_algebra}.

\subsection{Persistent Homology}\label{sec:persistent_homology}
We follow the exposition in \citep{chazal16_persistenceModules}. See also Chapter 3.3 of \citep{siu24_thesis}.

\begin{definition}[Filtration]
A filtration is a family $(X_t)_{t \in T}$ of topological spaces, where $t$ ranges over a subset $T$ of $\mathbb{R}$ (possibly the whole of $\mathbb{R}$), such that $X_s \subseteq X_t$ whenever $s, t \in T$ and $s \leq t$.
\end{definition}

\begin{example}[Sublevel Filtration]
Given a continuous function $f: X \to \mathbb{R}$ on a topological space $X$, its \emph{sublevel filtration} consists of the \emph{sublevel sets}
$$X_t = f^{-1}(-\infty, t] = \{x : f(x) \leq t\}$$
of $f$.
\end{example}

We will define persistent homology in terms of a persistence module.
\begin{definition}[Persistence Module; p.15 and p.17 of \citep{chazal16_persistenceModules}]
A family of vector spaces $(V_t)_{t \in T}$ parametrized over a subset $T$ of $\mathbb{R}$ with maps $\{L_{st}: V_s \to V_t | s, t \in T, s \leq t\}$ is said to be a persistence module if
\begin{itemize}
\item $L_{ss}$ is the identity map on $V_s$ for every $s$ and
\item $L_{rt} = L_{st}L_{rs}$ whenever $r < s < t$.
\end{itemize}
\end{definition}

\begin{definition}[Persistent Homology]
Let $q \geq 0$ and $\mathbb{F}$ be a field. The $q^{th}$ persistent homology of a filtration $(X_t)$ is the persistence module where $V_t = H_q(X_t; \mathbb{F})$, i.e. $V_t$ is the homology of $X_t$ with coefficients in $\mathbb{F}$ at dimension $q$ (defined in \cref{def:singular_homology}), and $L_{st}$ is the linear map induced by the inclusion map $X_s \to X_t$.
\end{definition}

We are now in the position to define the persistence diagrams of \emph{interval-decomposable} persistence modules, which includes the persistent homology of the sublevel filtrations of Morse functions on compact manifolds (Cf. \cref{prop:persistent_homology_of_morse_function_is_interval_decomposable}), and hence that of $f_{k,D}$ (Cf. the last point of \cref{lem:critical_points_of_tensor}).

Fix a field $\mathbb{F}$ and a subset $T$ of $\mathbb{R}$. A subset $J$ of $T$ is said to be an \emph{interval} if for every triple $r \leq s \leq t$ in $T$, $s \in J$ whenever $r, t \in J$. For an interval $J$, let 
$$
V(J)_t = \begin{cases} \mathbb{F} &\text{ if } t \in J\\
0 &\text{ otherwise,}
\end{cases}
\qquad \text{ and } \qquad
L(J)_{st} = \begin{cases} \text{identity on }\mathbb{F} & kj\text{ if } s, t \in J\\
0 &\text{ otherwise.}
\end{cases}
$$
The \emph{$J$-interval module} is the persistence module with vector spaces $V(J)_t$'s and linear maps $L(J)_{st}$'s.

Let $(V^i_t)$ be a family of persistence modules indexed by $i$ with maps $L^i_{st}$. The \emph{direct sum} $(\Oplus_i V^i_t)$ is the persistence module with vector spaces $(\Oplus_i V^i)_t$ and maps $\Oplus_i L^i_{st}$.

\begin{definition}[Interval Decomposition and Persistence Diagram; Section 2.5 -- 2.6 of \citep{chazal16_persistenceModules}]\label{def:persistence_diagram}
A persistence module $(V_t)$ is said to admit an \emph{interval decomposition} if it is isomorphic to the direct sum of interval modules $J_i$'s. In this case, the persistence diagram of $(V_t)$ is the multiset
$$\mathcal{D} = \{(\inf J_i, \sup J_i) \in [-\infty, \infty]^2: \inf J_i < \sup J_i\}.$$
A point in this multiset is called a persistence point (also known as homological feature). The coordinates of a persistence point $(b, d)$ are called its birth time ($b$) and the death time ($d$), and its persistence is defined as $d - b$.
\end{definition}
\begin{remark} $\quad$
\begin{enumerate}
\item A multiset is a set that can contain duplicate elements, and the duplicity is encoded by the multiplicity of the element, e.g. $\{0, 0, 1\}$ and $\{0, 1\}$ are different as multisets, because the multiplicity of $0$ in $\{0, 0, 1\}$ is 2. This can be formalized by defining a multiset as a set equipped with a multiplicity function, whose codomain is $\mathbb{N}_0 \cup \{+\infty\}$.
\item The Krull-Remak-Schmidt-Azumaya theorem (Theorem 2.7 of \citep{chazal16_persistenceModules}) guarantees that the persistence diagrams of interval-decomposable modules are well-defined.
\item For simplicity, we refer to the persistence diagrams of the persistent homology of the sublevel filtration of $f$ as the persistence diagrams of the sublevel filtration of $f$.
\end{enumerate}
\end{remark}

\begin{example}
Consider the following persistence module indexed by $\{1, 2, 3, 4, 5, 6\}$:

\begin{tikzcd}[ampersand replacement=\&]
V_1 \arrow[d, Rightarrow, no head]                         \& V_2 \arrow[d, Rightarrow, no head]                                   \& V_3 \arrow[d, Rightarrow, no head]                           \& V_4 \arrow[d, Rightarrow, no head]                  \& V_5 \arrow[d, Rightarrow, no head]                  \& V_6 \arrow[d, Rightarrow, no head]                  \\
\text{span }\{u_1\} \arrow[d, Rightarrow, no head]         \& {\text{span } \{v_1, v_2\}} \arrow[d, Rightarrow, no head]           \& {\text{span } \{w_1, w_2\}} \arrow[d, Rightarrow, no head]   \& \text{span } \{x_1\} \arrow[d, Rightarrow, no head] \& \text{span } \{y_3\} \arrow[d, Rightarrow, no head] \& \text{span } \{z_3\} \arrow[d, Rightarrow, no head] \\
\mathbb{F} \arrow[r, "\begin{bmatrix} 1\\0 \end{bmatrix}"] \& \mathbb{F}^2 \arrow[r, "{\begin{bmatrix} 1&0\\0&1 \end{bmatrix}}"] \& \mathbb{F}^2 \arrow[r, "{\begin{bmatrix} 1&0 \end{bmatrix}}"] \& \mathbb{F} \arrow[r, "0"]                           \& \mathbb{F} \arrow[r, "1"]                           \& \mathbb{F}                                          \\
u_1 \arrow[r, maps to]                                     \& v_1 \arrow[r, maps to]                                               \& w_1 \arrow[r, maps to]                                       \& x_1 \arrow[r, maps to]                              \& 0                                                   \&                                                     \\
                                                           \& v_2 \arrow[r, maps to]                                               \& w_2 \arrow[r, maps to]                                       \& 0                                                   \&                                                     \&                                                     \\
                                                           \&                                                                      \&                                                              \&                                                     \& y_3 \arrow[r, maps to]                              \& z_3                                                
\end{tikzcd}

Its persistence diagram is $\{(1, 4), (2, 3), (5, 6)\}$. The birth time of the point $(2, 3)$ is 2, and its death time is 3. Its persistence is 1.
\end{example}

We conclude with the discussion of filtrations and persistence modules that go in reverse direction. Such filtrations and persistence modules conform to the standard definitions above once we replace the parameter $t$ with $-t$, but we refrain from doing that for easier comparison between different filtrations that share the same parametrization. Respecting the direction of the reversed persistence module, the \emph{second} coordinate is called the birth time, and first is called the death time. The persistence is still their absolute difference.

\begin{example}
Consider the following reversed persistence module indexed by $\{1, 2, 3, 4, 5, 6\}$:
$$
\begin{tikzcd}[ampersand replacement=\&]
V_1 \arrow[d, Rightarrow, no head]                         \& V_2 \arrow[d, Rightarrow, no head]                                   \& V_3 \arrow[d, Rightarrow, no head]                           \& V_4 \arrow[d, Rightarrow, no head]                  \& V_5 \arrow[d, Rightarrow, no head]                  \& V_6 \arrow[d, Rightarrow, no head]                  \\
\text{span }\{u_1\} \arrow[d, Rightarrow, no head]         \& {\text{span } \{v_1, v_2\}} \arrow[d, Rightarrow, no head]           \& {\text{span } \{w_1, w_2\}} \arrow[d, Rightarrow, no head]   \& \text{span } \{x_1\} \arrow[d, Rightarrow, no head] \& \text{span } \{y_3\} \arrow[d, Rightarrow, no head] \& \text{span } \{z_3\} \arrow[d, Rightarrow, no head] \\
\mathbb{F} \& \arrow[l, "\begin{bmatrix} 1\\0 \end{bmatrix}"'] \mathbb{F}^2 \& \mathbb{F}^2 \arrow[l, "{\begin{bmatrix} 1&0\\0&1 \end{bmatrix}}"'] \& \mathbb{F} \arrow[l, "{\begin{bmatrix} 1&0 \end{bmatrix}}"']                           \& \mathbb{F} \arrow[l, "0"']                          \& \mathbb{F}  \arrow[l, "1"']                                          \\
u_1                                     \& v_1 \arrow[l, maps to]                                               \& w_1 \arrow[l, maps to]                                       \& x_1 \arrow[l, maps to]                              \&                                                   \&                                                     \\
                                                           0 \& v_2 \arrow[l, maps to]                                               \& w_2 \arrow[l, maps to]                                       \&                                                    \&                                                     \&                                                     \\
                                                           \&                                                                      \&                                                              \& 0                                                    \& y_3 \arrow[l, maps to]                              \& z_3 \arrow[l, maps to]                                                
\end{tikzcd}
$$

Its persistence diagram is $\{(1, 4), (2, 3), (5, 6)\}$. The birth time of the point $(2, 3)$ is 3, and its death time is 2. Its persistence is 1.
\end{example}

\begin{example}[Superlevel Filtration]
Given a continuous function $f: X \to \mathbb{R}$ on a topological space $X$, its \emph{superlevel filtration} consists of the \emph{superlevel sets} $$f^{-1}[t, \infty) = \{x : f(x) \geq t\}$$
of $f$. Note that for $s \leq t$, $f^{-1}[s, \infty)\supseteq f^{-1}[t, \infty)$.
\end{example}

\begin{example}[Persistent Cohomology]
Let $q \geq 0$ and $\mathbb{F}$ be a field. The $q^{th}$ persistent cohomology of a filtration $(X_t)$ is the persistence module where $V_t = H^q(X_t; \mathbb{F})$ (superscript denotes cohomology, cohomology is defined in \cref{def:cohomology}) and $L_{ts}: H^q(X_t; \mathbb{F}) \to H^q(X_s; \mathbb{F})$ is the linear map induced by the inclusion map $X_s \to X_t$. Note the reversed direction of $L_{st}$.
\end{example}

\subsection{Morse Theory}

We first recall the basic terminology of Morse theory.
Let $M$ be a compact Riemannian manifold of dimension $D < \infty$. Let $p \in M$ and $f: M \to \mathbb{R}$ be a smooth function. Let $(x_1, ..., x_D)$ be a local coordinate system at $p$. The point $p$ is said to be a \emph{critical point} if the differential $df_p: T_p M \to \mathbb{R}$ of $f$ at $p$ vanishes, i.e. $df_p = 0$, or equivalently, if $\frac{\partial f}{\partial x_1} = ... = \frac{\partial f}{\partial x_D} = 0$. If the point $p$ is critical, it is said to be \emph{nondegenerate} if the Hessian matrix $\frac{\partial^2 f}{\partial x_i \partial x_j}$ is nonsingular. Note that at a critical point, the Hessian is a well-defined bilinear form on the tangent space, and hence its nondegeneracy is independent of the choice of the coordinate system (Cf. Chapter 2 of \citep{milnor63_morseTheory}). If the point $p$ is a nondegenerate critical point, its \emph{index} is the number of negative eigenvalues of the Hessian matrix. The function $f$ is said to be a \emph{Morse function} if all its critical points are nondegenerate.

A critical value is a function value of a critical point. A regular value is a real number that is not a critical value.

We now define the Morse complex, which is a chain complex (Cf. \cref{def:chain_complex}). It is the main tool to compute the homology of sublevel sets of a Morse function. We denote by $M^t = \{p \in M: f(p) \leq t\}$ the sublevel sets of $f$.

\begin{definition}[Morse Complex with Coefficients in $\mathbb{Z}/2\mathbb{Z}$; Cf. Section 2.4 of \citep{hutchings02_morseTheory} and Chapter 3.1 of \citep{audin14_morseTheory}]\label{def:morse_complex_mod_2}
The Morse complex $C^\text{Morse}_\bullet(f, t; \mathbb{Z}/2\mathbb{Z})$ of the sublevel set $M^t$ with coefficients in $\mathbb{Z}/2\mathbb{Z}$ is defined as follows.
$$C^\text{Morse}_q(f, t; \mathbb{Z}/2\mathbb{Z}) = \left\{\sum a_i p_i: a_i \in \mathbb{Z}/2\mathbb{Z}, p_i \text{ critical with index }q, f(p_i) \leq t\right\},$$ and
$$\partial_q p = \sum_{p' \text{ critical with index } q-1} n(p, p') p',$$
where $n(p, p')$ is the number of gradient descent flow lines from $p$ to $p'$ modulo 2.
\end{definition}
\begin{remark}
Since a Morse function on a compact manifold has finitely many critical points (by Corollary 2.3 of \cite{milnor63_morseTheory}), the sum in the definition of $C^\text{Morse}_q(f, t; \mathbb{Z}/2\mathbb{Z})$ is always a finite sum.
\end{remark}

\begin{theorem}\label{thm:morse_homology_naturally_isomorphic_to_singular_homology}
Let $f: M \to \mathbb{R}$ be a Morse function satisfying the Smale condition (defined in \cref{def:smale_condition}). Denote by $M^t = \{p \in M: f(p) \leq t\}$ and $M^t_s = \{p \in M: s \leq f(p)\leq t\}$. Denote by $H_*$ and $H^\text{Morse}_*$ the singular (Cf. p.108 of \citep{hatcher02_algtopo}) and Morse homology. Suppose $s < t$, and they are both regular values of $f$. Then there exist canonical isomorphisms between singular and Morse homology such that following diagram
\begin{equation*}
\begin{tikzcd}
H_*(M^s; \mathbb{Z}/2\mathbb{Z}) \ar[r] \ar[d, "\cong"]
& H_*(M^t; \mathbb{Z}/2\mathbb{Z}) \ar[d, "\cong"]\\
H^\text{Morse}_*(M^s; \mathbb{Z}/2\mathbb{Z}) \ar[r]
& H^\text{Morse}_*(M^t; \mathbb{Z}/2\mathbb{Z}),
\end{tikzcd}
\end{equation*}
where horizontal maps are induced by inclusion, commutes. Further, the analogous theorem for reduced homology holds.
\end{theorem}

\begin{proof}[Sketch of Proof]
This is a straight-forward consequence of Theorem 3.2 of \citep{hutchings02_morseTheory}. We first show that Theorem 3.2 of \citep{hutchings02_morseTheory} implies the integral version of the theorem (with $\mathbb{Z}/2\mathbb{Z}$ replaced by $\mathbb{Z}$ throughout). To see this, first, fixing $s$ and applying Theorem 3.2 of \citep{hutchings02_morseTheory} to $g_s = (f - \min f + 1/2)/(s - \min f + 1/2)$ shows that $H_*(M^s) \cong H^\text{Morse}_*(M^s)$ through a canonical map ($X_0 = g_s^{-1}\{0\}$ is empty, $X_1 = g_s^{-1}\{1\} = \{p \in M: f(p) = s\}$). The canonical map, explicitly defined in the proof, sends a simplex to the critical points that ``catch" the simplex under gradient descent flow, and this map commutes with inclusion. The claim for absolute homology then follows.

For the reduced case, one may define the reduced homology as the relative homology with a global minimum, and the canonical isomorphism follows by taking the quotient with the subspace generated by the global minimum.

The $\mathbb{Z}/2\mathbb{Z}$ version follows from a standard application of the universal coefficient theorem for homology and the five lemma (Theorem 3A.3 and p.129 of \citep{hatcher02_algtopo}).
\end{proof}

In light of this theorem, we will not distinguish between singular homology and Morse homology from this point onwards.

\begin{proposition}\label{prop:persistent_homology_of_morse_function_is_interval_decomposable}
Let $f: M \to \mathbb{R}$ be a Morse function on a smooth Riemannian manifold $M$. Then the persistent homology with coefficients in $\mathbb{Z}/2\mathbb{Z}$ of the sublevel filtration of $f$ is interval-decomposable.
\end{proposition}

\begin{proof}
By Theorem 2.8 of Chazal, it suffices to show that $H_*(M^t; \mathbb{Z}/2\mathbb{Z})$ is finitely generated for every $t$. By 
\cref{thm:morse_homology_naturally_isomorphic_to_singular_homology}, $H_*(M^t; \mathbb{Z}/2\mathbb{Z}) \cong H^\text{Morse}_*(M^t; \mathbb{Z}/2\mathbb{Z})$. Since each $C^\text{Morse}_*(M^t)$ is finitely generated (Cf. the remark after \cref{def:morse_complex_mod_2}), so is each of $H^\text{Morse}_*(M^t)$. The proposition then follows.
\end{proof}

\subsection{Computational Topology}

In this subsection, we present the main algorithm to compute persistence diagrams. We deviate from standard notation in the literature to explicitly highlight the role of the ordering of basis elements in chain groups, so that we can carry out careful bookkeeping in our main induction argument in \cref{prop:column_reduce_boundary}.

\begin{definition}[Concatenated Complex and Concatenated Differential]
Let $(C_\bullet, \partial_\bullet)$ be a chain complex of vector spaces. The concatenated complex is defined as $C = \oplus_{p \in \mathbb{Z}} C_p$. The concatenated differential $\partial: C \to C$ is defined by
$$\partial c = \partial_p c$$ whenever $c \in C_p$.
\end{definition}

\begin{definition}[Lowest Generator and Reduced Matrix]
Let $\{w_1, ..., w_n\}$ be a basis of $(\mathbb{Z}/2\mathbb{Z})^n$. Suppose this basis is totally ordered by $\prec$.
\begin{enumerate}
\item For a nonzero vector $v \in (\mathbb{Z}/2\mathbb{Z})^n$, define
$\text{low}_\prec(v)$ as the first $w_j$ (with respect to $\prec$) such that $v \in \text{span } \{w_i: w_i \prec w_j \text{ or } w_i = w_j\}$. We call $\text{low}_\prec(v)$ the \emph{lowest generator} of $v$.
\item Let $R: (\mathbb{Z}/2\mathbb{Z})^n \to (\mathbb{Z}/2\mathbb{Z})^n$ be a linear map. $R$ is said to be reduced with respect to $\prec$ if $\text{low}_\prec(R w_i) \neq \text{low}_\prec(R w_j)$ whenever $i \neq j$ and $R w_i$ and $R w_j$ are both nonzero.
\end{enumerate}
\end{definition}

\begin{definition}[Upper Triangular]
Let $\{w_1, ..., w_n\}$ be a basis of $(\mathbb{Z}/2\mathbb{Z})^n$ that is totally ordered by $\prec$. A linear map $U:(\mathbb{Z}/2\mathbb{Z})^n \to (\mathbb{Z}/2\mathbb{Z})^n$ is said to be upper triangular with respect to $\prec$ if $Uw_i \in \text{span } (\{w_i\} \cup \{w_j: w_j \prec w_i \})$. It is said to be unit upper triangular if $Uw_i \in w_i + \text{span } \{w_j: w_j \prec w_i \}$.
\end{definition}

\begin{definition}[Column Reduced Matrix]
Let $\{w_1, ..., w_n\}$ be a basis of $(\mathbb{Z}/2\mathbb{Z})^n$ that is totally ordered by $\prec$, and $\partial, V, R: (\mathbb{Z}/2\mathbb{Z})^n \to (\mathbb{Z}/2\mathbb{Z})^n$ be linear maps. Suppose $\partial^2 = 0$. $\partial$ is said to be column reduced via 
\begin{equation}\label{eqn:column_reduced}
R = \partial V
\end{equation}
with respect to $\prec$ if the followings hold.
\begin{description}
\item[Factorization] \cref{eqn:column_reduced} holds when the right-hand side is interpreted as composition of linear maps (in matrix form: as matrix multiplication).
\item[Unit-Upper-Triangular Condition] $V$ is unit upper triangular with respect to $\prec$.
\item[Reduction] $R$ is reduced with respect to $\prec$.
\end{description}
\end{definition}

\begin{theorem}\label{thm:persistence_diagram_from_reduced_matrix}
Let $f: M \to \mathbb{R}$ be a smooth Morse function on a compact Riemannian manifold $M$. Let $C_p$ be the $p^{th}$ Morse chain group of $f$ and $\partial: \oplus_p C_p \to \oplus_p C_p$ be the concatenated differential. Let $\prec$ be a total order on the critical points of $f$ such that $c \prec c'$ whenever $f(c) < f(c')$. Suppose $\partial$ is column reduced via $R = \partial V$. Then the persistence diagram of the sublevel filtration of $f$ at dimension $p$ is given by $\mathcal{D}_p^\text{finite} \cup \mathcal{D}_p^\text{infinite}$, where
\begin{align*}
\mathcal{D}_p^\text{finite} &= \{(f(\text{low}_\prec (Rc)), f(c)) : \text{ind } c = p+1, Rc \neq 0\}
\\
\mathcal{D}_p^\text{infinite} &= \{(f(c), \infty) : \text{ind } c = p, Rc = 0, c \neq \text{low}_\prec(Rc') \text{ for every critical point } c'\}.
\end{align*}
\end{theorem}

\begin{remark}
Note that $\prec$ is not necessarily unique.
\end{remark}

\begin{proof}[Sketch of Proof]
By \cref{thm:morse_homology_naturally_isomorphic_to_singular_homology}, it suffices to compute the persistent homology of the Morse complex. Since the Morse complex is finitely generated, it is amenable to the algorithmic computation described in Chapter 3.3 of \citep{dey22_computationalToplogy}. By the paragraph above Fact 3.9 of \citep{dey22_computationalToplogy}, it suffices to column reduce $\partial$ to obtain the so-called persistence pairs from the lowest generators. This gives the claim about $\mathcal{D}_p^\text{finite}$. $\mathcal{D}_p^\text{infinite}$ consists of points of the form $(f(c), \infty)$, where $c$ is unpaired. By Theorem 3.6 of \citep{dey22_computationalToplogy}, if $Rc \neq 0$, then $c$ is necessarily paired with $\text{low}(Rc)$. Therefore, if $c$ is unpaired, $Rc = 0$. The claim about $\mathcal{D}_p^\text{infinite}$ then follows.
\end{proof}

\begin{lemma}\label{lem:R_of_lowest_one_must_vanish}
Let $\{w_1, ..., w_n\}$ be a basis of $(\mathbb{Z}/2\mathbb{Z})^n$ that is totally ordered by $\prec$, and $\partial: (\mathbb{Z}/2\mathbb{Z})^n \to (\mathbb{Z}/2\mathbb{Z})^n$ be a linear map. Suppose $\partial^2 = 0$, and $\partial$ is column reduced via $R = \partial V$ with respect to $\prec$. If $w_i = \text{low}_\prec(Rw_j)$, then $Rw_i = 0$.
\end{lemma}

\begin{proof}
If $w_i = \text{low}_\prec(Rw_j)$, then $R w_j \in w_i + w$ for some $w \in \text{span } \{w_\ell: w_\ell \prec w_i\}$. This means $w_i + w \in \text{im } R \subseteq \text{im } \partial \subseteq \ker \partial$, where the first inclusion follows from $R = \partial V$, and the second follows from $\partial^2 = 0$.
Then we have the following sequence of equations:
\begin{align*}
\partial (w_i + w) &= 0
\\R V^{-1}(w_i + w) &= 0
\\R (w_i + v' + V^{-1}w) &= 0 \text{ for some } v' \in \text{span } \{w_\ell: w_\ell \prec w_i\}
\\R (w_i + v) &= 0 \text{ for some } v \in \text{span } \{w_\ell: w_\ell \prec w_i\},
\end{align*}
where the second last line follows from the fact that $V$ is upper triangular, and the last line follows by letting $v = v' + V^{-1}w$. The last line implies
$$Rw_i \in \text{span} \{Rw_\ell: w_\ell \prec w_i\} = \text{span} \{Rw_\ell: w_\ell \prec w_i, \quad Rw_\ell \neq 0\}.$$ Since $R$ is reduced, distinct nonzero images have distinct lowest generators, and hence $\{Rw_\ell: w_\ell \prec w_i, Rw_\ell \neq 0\}$ is linearly independent. This forces that $Rw_i$ must be 0.
\end{proof}

\section{Morse Theoretic Computations}
\label{sec:morse_computation}


First, we set up the notation to describe the critical points of $f_{k,D}$.

For each signum vector $\xi \in \Xi = \{-1, 0, 1\}^D - \{(0, ..., 0)\}$, define $x^\xi \in S^{D-1}$ by 
\begin{align}\label{eqn:critical_points}
x^\xi_i &\propto \xi_i \lambda_i^{-\frac{1}{k-2}}
\end{align}
For $\xi \in \Xi$, if $\{\xi_i^k : \xi_i \neq 0\}$ is a singleton, we denote the element of the singleton by $\xi^k$; otherwise, $\xi^k$ is undefined.

\begin{lemma}\label{lem:critical_points_of_tensor} $\quad$
\begin{enumerate}
\item The set of critical points of $f_{k,D}$ is
\begin{equation*}
\mathcal{C} = \begin{cases}
\{x^\xi : \xi \in \Xi\} & \text{ if $k$ is even,}
\\
\{x^\xi : \xi \in \Xi, \xi_i \geq 0 \text{ for all $i$, or }\xi_i \leq 0 \text{ for all $i$}\} & \text{ if $k$ is odd}.
\end{cases}
\end{equation*}
\item Whenever $x^\xi \in \mathcal{C}$, $\xi^k$ is well-defined.
\item For $x^\xi \in \mathcal{C}$, the Hessian of $f_{k,D}$ at $x^\xi$, in ambient coordinates, is $$\frac{\xi^k}{\|\lambda^{-\frac{1}{k-2}}\|_{L^2(|\xi|)}^{k-2}} [(k-1) P^\xi - I],$$ where $I$ is the identity matrix, and $P^\xi$ is the linear operator that projects $\mathbb{R}^D$ to $\text{span } \{e_i : \xi_i \neq 0\}$.
\item Every critical point $x^\xi \in \mathcal{C}$ is nondegenerate, and its index is
$$
\begin{cases}
D - |\{i: \xi_i \neq 0\}| & \text{ if } \xi^k = 1
\\
|\{i: \xi_i \neq 0\}| - 1 & \text{ if } \xi^k = -1
\end{cases}
$$
\item $f_{k,D}$ is a Morse function.
\end{enumerate}
\end{lemma}

\begin{remark}
We prove in \cref{thm:f_is_smale} that $f_{k,D}$ satisfies the Smale condition.
\end{remark}

\begin{proof}
The proof of the first point is a straight-forward Lagrange multiplier argument. The parity of $k$ matters because $t \mapsto t^{k-2}$ is bijective for odd $k$ but not for even $k$.

For the second point, if $k$ is odd, the $\xi^k$ is well-defined by the construction of $\mathcal{C}$; if $k$ is even, then $1^k = (-1)^k = 1$, and hence $(x^\xi_i)^k = 1$. The remainder of the second point is a straight-forward computation.

For the third point, consider a unit tangent vector $v$ of $S^{D-1}$ at $x^\xi$, and consider the curve $\gamma(t) = (\cos t)x + (\sin t) v$ on $S^{D-1}$ so that $\gamma(0) = x$ and $\gamma'(0) = v$.
Then Taylor expansion gives
\begin{align*}
f_{k,D}(\gamma(t))
&=\frac{1}{k}\sum \lambda_i (x^\xi_i \cos t + v_i \sin t )^k
\\&=\frac{1}{k} \sum \lambda_i (x^\xi_i + v_i t - \frac{x^\xi_i}{2}t^2 + O(t^3) )^k
\\&= f_{k,D}(x^\xi) + t \nabla f_{k,D}(x^\xi) \cdot v + \frac{t^2}{2} \sum \lambda_i (x^\xi_i)^{k-2}[(k-1) v_i^2 - (x^\xi_i)^2] + O(t^3).
\end{align*}
By the definition of $x^\xi$, if $\xi_i \neq 0$, then $\lambda_i (x^\xi_i)^{k-2} = \frac{\xi^{k-2}}{\|\lambda^{-\frac{1}{k-2}}\|_{L^2(|\xi|)}^{k-2}} = \frac{\xi^{k}}{\|\lambda^{-\frac{1}{k-2}}\|_{L^2(|\xi|)}^{k-2}}$. Thus,
$$f_{k,D}(\gamma(t)) = f_{k,D}(x^\xi) + t \nabla f_{k,D}(x^\xi) \cdot v + \frac{t^2}{2} \frac{\xi^{k}}{\|\lambda^{-\frac{1}{k-2}}\|_{L^2(|\xi|)}^{k-2}} [(k-1)\sum_{\xi_i \neq 0} v_i^2 - 1].$$
The third point then follows from the fact that the expression of the Hessian therein is the unique bilinear form that gives rise to the quadratic term above.
For the fourth point, recall that the tangent plane of the unit sphere at $x^\xi$ is $(x^\xi)^\perp$. If $\xi^k = 1$, then the Hessian is positive-definite on $(x^\xi)^\perp \cap \text{span }\{e_i : \xi_i \neq 0\}$ and it is negative-definite on $\text{span }\{e_i: \xi_i = 0\} \subseteq (x^\xi)^\perp$. The dimensions of these two spaces are $|\{i: \xi_i \neq 0\}| - 1$ and $D - |\{i: \xi_i \neq 0\}|$ respectively. The claim then follows.
The last point follows directly from the fourth point.
\end{proof}

Next, we describe the critical values of $f_{k,D}$ and their ordering on the real line.

\begin{lemma} \label{lem:critical_values_and_ordering}$\quad$
\begin{enumerate}
\item For each critical point $x^\xi \in \mathcal{C}$ of $f_{k,D}$, $f_{k,D}(x^\xi) = \xi^k \nu_\xi$.
\item For critical points $x^\xi, x^\zeta \in \mathcal{C}$, if $\xi^k = 1$ and $\zeta^k = -1$, then $f_{k,D}(x^\xi) > 0 > f_{k,D}(x^\zeta)$.
\end{enumerate}
\end{lemma}

We now describe the differential of the Morse complex of $f_{k,D}$ with coefficients in $\mathbb{Z}/2\mathbb{Z}$ for even $k$ (Cf \cref{def:morse_complex_mod_2}). 

\begin{proposition}\label{prop:morse_differential_tensor}
If $\xi^k = \zeta^k = 1$, then
$$n(x^\xi, x^\zeta) = \begin{cases}
1 & \text{ if } \zeta = \xi + s e_i \text{ for some } s \in \{1, -1\} \text{ and some $i$ such that } \xi_i = 0 \\
0 & \text{ otherwise.}
\end{cases}$$
If $\xi^k = -1$, then
$$n(x^\xi, x^\zeta) = \begin{cases}
1 & \text{ if } \xi = \zeta - e_i \text{ for some $i$ such that } \zeta_i = 0 \\
0 & \text{ otherwise.}
\end{cases}$$
Hence, if $k$ is even, then
$$\partial x^\xi  = \sum_{\substack{1 \leq i \leq D \\ \xi_i = 0}} (x^{e_i + \xi} + x^{-e_i + \xi});$$
if $k$ is odd, $\xi^k = -1$, then
\begin{equation}\label{eqn:morse_differential_odd}
\partial x^\xi  = 
\begin{cases}
\sum_{\substack{1 \leq i \leq D \\ \xi_i = -1}} x^{e_i + \xi} & \text{ if } \{i: \xi_i \neq 0\} \text{ has more than one element,}\\
0 &\text{ otherwise.}
\end{cases}
\end{equation}
\end{proposition}

We need the following lemma to prove \cref{prop:morse_differential_tensor}.
\begin{lemma}\label{lem:never_descent_to_zero}
Consider the gradient descent of $f_{k,D}$ on $S^{D-1}$:
$$x' = - [\nabla f_{k,D}(x) - (\nabla f_{k,D}(x) \cdot x) x],$$
where the last term projects the gradient back to the tangent space.
Let $x^\zeta = \lim_{t \to +\infty} x(t)$ and $x^\xi = \lim_{t \to -\infty} x(t)$. 
\begin{enumerate}
\item If, for some $i$, $\xi_i$ and $\zeta_i$ are both nonzero, then they have the same sign.
\item If $\zeta^k = 1$ and $\zeta_i = 0$, then $\xi_i = 0$. Further, $x_i(t) \equiv 0$.
\item If $\xi^k = -1$ and $\xi_i = 0$, then $\zeta_i = 0$. Further, $x_i(t) \equiv 0$.
\end{enumerate}
\end{lemma}

\begin{proof}
The first claim follows from the fact that each $\{x \in S^{D-1} : x_i = 0\}$ is invariant under the gradient flow of $f_{k,D}$.



For the second claim, evaluating $\nabla f_{k,D}$ gives
$$x_i' = (k f_{k,D}(x) - \lambda_i x_i^{k-2}) x_i.$$
Since $x^\zeta_i = 0$, we have
$$x_i' = (k f_{k,D} (x^\zeta) + o(1)) x_i$$
as $t \to +\infty$. Since $f_{k,D}(x^\zeta) > 0$, either $x_i \equiv 0$, or $x_i$ contradictorily grows exponentially on the unit sphere. The second claim then follows from the invariance of $\{x \in S^{D-1} : x_i = 0\}$ under the gradient flow.

The third claim is similar.
\end{proof}

We now prove \cref{prop:morse_differential_tensor}.
\begin{proof}[Proof of \cref{prop:morse_differential_tensor}]

We first prove the first claim (assuming $\xi^k = \zeta^k = 1$).

We first show the ``otherwise" statement. Suppose $x^\xi$ and $x^\zeta$ are in adjacent chain vector spaces, and there is at least one flow line from $x^\xi$ to $x^\zeta$. We will show that $\zeta$ and $\xi$ have the desired form.

Let $\text{supp }\zeta = \{i: \zeta_i \neq 0\}$ and $\text{supp }\xi = \{i: \xi_i \neq 0\}$. Then the second point of \cref{lem:never_descent_to_zero} shows that $\text{supp } \xi \subseteq \text{supp } \zeta $. By the fourth point in \cref{lem:critical_points_of_tensor}, for $x^\xi$ and $x^\zeta$ to be in adjacent chain vector spaces, $\zeta$ must have exactly one more nonzero entry than $\xi$. Hence there is a unique $i \in \text{supp } \zeta - \text{supp } \xi$. The first point of \cref{lem:never_descent_to_zero} implies that for $j \in \text{supp }\xi$, $\xi_j$ and $\zeta_j$ in fact have the same sign, and hence they must be equal (because they are $\pm 1$). The ``otherwise" statement then follows.

To prove the proposition, it remains to establish that there is exactly one flow line from $x^\xi$ to $x^\zeta$ if they are of the form prescribed by the proposition. We prove this by induction on the ambient dimension $D$. If $D = 1$, the proposition is vacuously true. Suppose the proposition is true for dimension $D-1$. If $\zeta_i = 0$ for some $i$, then $x^\xi, x^\zeta$ lie in the $(D-2)$-dimensional unit sphere $S^{D-2}$ in $\text{span } \{e_1, ..., e_{i-1}, e_{i+1}, ..., e_D\}$, and the second point of \cref{lem:never_descent_to_zero} implies that any gradient flow line between them lies in $S^{D-2}$ as well. Therefore, the number of downward gradient flow lines between them on $S^{D-1}$ is given by the induction hypothesis.  

It remains to consider the case when $\zeta_i \neq 0$ for every $i$. Then $\xi_i = 0$ for exactly one $i$. By the fourth point in \cref{lem:critical_points_of_tensor}, $x^\zeta$ is a local minimum and the index of $x^\xi$ is 1. The form of the Hessian at $x^\xi$ (Cf. the third point in \cref{lem:critical_points_of_tensor}) implies that there are two descent flow lines from $x^\xi$, one tangent to $e_i$ and the other one tangent to $-e_i$. Since
$$\{x \in S^{D-1}: x_i \geq 0, \text{ and for $j \neq i$, $x_j$ and $\xi_j$ have the same (non-strict) sign}\}$$
is invariant under gradient descent, the flow line tangent to $e_i$ converges to a critical point $\hat \zeta$ in this set. By the two points in \cref{lem:never_descent_to_zero}, $\hat \zeta_j = \xi_j$ for $j \neq i$. Since this limit cannot be $\xi$ itself, it can only be $\xi + e_i$. The case for $\xi - e_i$ is similar.

The second claim is largely similar with the first claim. The only difference is in considering the case when $\xi_j \neq 0$ for every $j$ with $\zeta = \xi + e_i$, parallel to the previous paragraph. In this case, $\zeta$ is index $D-2$ and has one ascent flow line into $\{x \in S^{D-1}: x_i \leq 0\}$ and one out of it. By the first claim of \cref{lem:never_descent_to_zero}, the latter cannot flow back to $\xi$. The former flows to $\xi$ by a sign invariance argument analogous to the one in the previous paragraph.

The remaining claims of the proposition directly follow from the first two claims.
\end{proof}

\section{Recurrence on Morse Homology for Even $k$}
\label{sec:morse_recurrence_even}

Let $C_D$ be the concatenated complex of the Morse complex of $ f_{k,D}$ on $S^{D-1}$ for an even $k$, with concatenated differential $\partial_D: C_D \to C_D$. Let 
\begin{align*}
C_{D-1}^+ &= \{x^\xi \in C_D: \xi_D = 1\}\\
C_{D-1}^0 &= \{x^\xi \in C_D: \xi_D = 0\}\\
C_{D-1}^- &= \{x^\xi \in C_D: \xi_D = -1\}.
\end{align*}
Then $C_D$ admits the recurrence $C_D = C_{D-1}^+ \cup C_{D-1}^- \cup C_{D-1}^0 \cup \{x^{e_D}, x^{-e_D}\}$, where $C_{D-1}$ is in bijective correspondence with $C_{D-1}^+$ through $x^\xi \xmapsto{I_{D-1}^+} x^{\xi + e_D}$. The bijection $I_{D-1}^-$ is defined analogously, and these bijections extend linearly to isomorphisms between the spans of their domains and codomains.

\begin{proposition}\label{prop:morse_differential_block_structure}
Let $u_D = \sum_{1 \leq i \leq D} (x^{e_i} + x^{-e_i})$, and $u^+_{D-1} = I^+_{D-1} u_{D-1}$ and $u^-_{D-1} = I^-_{D-1} u_{D-1}$. Then the concatenated differential $\partial_D$ has the following block structure:
\[
\begin{blockarray}{cccccc}
  & C_{D-1}^+ & C_{D-1}^- & C_{D-1}^0 & x^{e_D} & x^{-e_D} \\
\begin{block}{c[ccccc]}
  C_{D-1}^+    & I^+_{D-1} \partial_{D-1} (I^+_{D-1})^{-1} &                                   & I^+_{D-1}          & u^+_{D-1} &     \\
  C_{D-1}^-    &                                   & I^-_{D-1} \partial_{D-1} (I^-_{D-1})^{-1} & I^-_{D-1}          &     & u^-_{D-1} \\
  C_{D-1}^0    &                                   &                                   & \partial_{D-1} &     &       \\
  x^{e_D}      &                                   &                                   &                &     &       \\
  x^{-e_D}     &                                   &                                   &                &     &       \\
  \end{block}
\end{blockarray},
\]
where empty entries are zero, and we identify $C_{D-1}^0$ with $C_{D-1}$.
\end{proposition}

\begin{proof}
We check the block structure holds column by column.
For the first column, let $x^\xi \in C_{D-1}^+$ and $\zeta = (I_{D-1}^+)^{-1} x^\xi = \xi - e_D$. Then
\begin{align*}
\partial_D x^\xi
&= \sum_{\substack{1 \leq i \leq D \\ \xi_i = 0}} x^{e_i + \xi} + x^{-e_i + \xi}
\\&= \sum_{\substack{1 \leq i \leq D-1 \\ \xi_i = 0}} x^{e_i + \xi} + x^{-e_i + \xi} \qquad \text{(because $\xi_D = 1 \neq 0$)}
\\I^+_{D-1} \partial_{D-1} (I^+_{D-1})^{-1} x^\xi
&= I^+_{D-1} \partial_{D-1} \zeta
\\&= I^+_{D-1} \sum_{\substack{1 \leq i \leq D-1 \\ \zeta_i = 0}} x^{e_i + \zeta} + x^{-e_i + \zeta}
\\&= \sum_{\substack{1 \leq i \leq D-1 \\ \zeta_i = 0}} x^{e_i + \zeta + e_D} + x^{-e_i + \zeta + e_D}
\\&= \sum_{\substack{1 \leq i \leq D-1 \\ \zeta_i = 0}} x^{e_i + \xi} + x^{-e_i + \xi},
\end{align*}
which is the same as $\partial_D x^\xi$ because for $1 \leq i \leq D-1$, $\xi_i = 0$ if and only if $\zeta_i = 0$.
The second column is analogous. For the third column, let $\xi_D = 0$. Then
\begin{align*}
\partial_D x^\xi
&= \sum_{\substack{1 \leq i \leq D \\ \xi_i = 0}} x^{e_i + \xi} + x^{-e_i + \xi}
\\&= x^{e_D + \xi} + x^{-e_D + \xi} + \sum_{\substack{1 \leq i \leq D-1 \\ \xi_i = 0}} x^{e_i + \xi} + x^{-e_i + \xi}
\\&= I_{D-1}^+ x^\xi + I_{D-1}^- x^\xi + \partial_{D-1} x^\xi.
\end{align*}
For the fourth column, $$\partial_D x^{e_D} = \sum_{1 \leq i \leq D-1} (x^{e_i + e_D} + x^{- e_i + e_D}) = \sum_{1 \leq i \leq D-1} (I_{D-1}^+ x^{e_i} + I_{D-1}^+ x^{- e_i}) = I_{D-1}^+ u_{D-1}.$$
The fifth column is analogous.
\end{proof}

Before column reducing $\partial_D$, we first order $C_D$, and we establish a lemma about lowest generators.

\begin{definition}\label{def:order_on_chain_group}
The total order $\prec$ on $C_D \cup \{-1, 0, 1\}$ is defined as follows.
\begin{enumerate}
\item $x^\xi \prec x^\zeta$ whenever $f_{k,D}(x^\xi) < f_{k,D}(x^\zeta)$,
\item $1 \prec -1 \prec 0$, 
\item If $\nu_\xi = \nu_\zeta$, then the two are ordered lexicographically, i.e. $x^\xi \prec x^\zeta$ if $\xi_i \prec \zeta_i$, where $i$ is the first entry (in the usual ordering of $1, ..., D$) where $x^\xi$ and $x^\zeta$ differ.
\end{enumerate}
\end{definition}

The lemma below summarizes the only properties we need about $\prec$ defined above.
\begin{lemma} The following statements are true regarding $\prec$.
\begin{enumerate}
\item $e_1 \prec -e_1$
\item $I^+_D$ and $I^-_D$ are monotone with respect to $\prec$ for every $D \geq 1$.
\item For each $x^\xi \in C_{D-1}$, $I^+_{D-1}(x^\xi) \prec I^-_{D-1}(x^\xi)$.
\item $x^\xi \prec x^\zeta$ whenever $f_{k,D}(x^\xi) < f_{k,D}(x^\zeta)$.
\end{enumerate}
\end{lemma}

\begin{proof}
Direct verification.
\end{proof}

\begin{lemma}\label{lem:low_boundary}
Let $x^\xi \in C_D$. Let $\xi_i$ be the first zero entry of $\xi$. Then
$\text{low}(\partial_D x^\xi) = x^{-e_i + \xi}$.
\end{lemma}
\begin{proof}
$\partial_D x^\xi = \sum_{\xi_j = 0} x^{e_j + \xi} + x^{-e_j + \xi}$. The critical points $x^{\pm e_i + \xi}$ attain the highest function values. By our convention, $x^{-e_i + \xi}$ is the lowest generator.
\end{proof}

Now, we column reduce $\partial_D$ inductively to establish a recurrence concerning the persistence diagrams of $T_k$ on $S^{D-1}$. Our goal is the last two points of \cref{prop:column_reduce_boundary}.

First, we define $V_2$ and $R_2$ as follows.
\begin{align}
\partial_2: &
\begin{blockarray}{cccccccccc}
  && +&-&+&-&+&-&0&0 \\
  && +&+&-&-&0&0&+&- \\
\begin{block}{cc[cccccccc]}
  +&+    & 0 &   &   &   & 1 &   & 1 &   \\
  -&+    &   & 0 &   &   &   & 1 & 1 &   \\
  +&-    &   &   & 0 &   & 1 &   &   & 1 \\
  -&-    &   &   &   & 0 &   & 1 &   & 1 \\
  +&0    &   &   &   &   & 0 &   &   &   \\
  -&0    &   &   &   &   &   & 0 &   &   \\
  0&+    &   &   &   &   &   &   & 0 &   \\
  0&-    &   &   &   &   &   &   &   & 0 \\
  \end{block}
\end{blockarray} \notag
\\R_2: &
\begin{blockarray}{cccccccccc}
  && +&-&+&-&+&-&0&0 \\
  && +&+&-&-&0&0&+&- \\
\begin{block}{cc[cccccccc]}
  +&+    & 0 &   &   &   & 1 &   & 1 &   \\
  -&+    &   & 0 &   &   &   &   & 1 &   \\
  +&-    &   &   & 0 &   & 1 &   &   & 1 \\
  -&-    &   &   &   & 0 &   &   &   & 1 \\
  +&0    &   &   &   &   & 0 &   &   &   \\
  -&0    &   &   &   &   &   & 0 &   &   \\
  0&+    &   &   &   &   &   &   & 0 &   \\
  0&-    &   &   &   &   &   &   &   & 0 \\
  \end{block}
\end{blockarray} \notag
\\
V_2: &
\begin{blockarray}{cccccccccc}
  && +&-&+&-&+&-&0&0 \\
  && +&+&-&-&0&0&+&- \\
\begin{block}{cc[cccccccc]}
  +&+    & 1 &   &   &   &   &   &   &   \\
  -&+    &   & 1 &   &   &   &   &   &   \\
  +&-    &   &   & 1 &   &   &   &   &   \\
  -&-    &   &   &   & 1 &   &   &   &   \\
  +&0    &   &   &   &   & 1 & 1 &   &   \\
  -&0    &   &   &   &   &   & 1 &   &   \\
  0&+    &   &   &   &   &   & 1 & 1 &   \\
  0&-    &   &   &   &   &   & 1 &   & 1 \\
  \end{block}
\end{blockarray},
\end{align}
where the signs outside of the matrices refer to the signs of the basis elements corresponding to their rows and columns, e.g. the second column and the second row correspond to $x^{(-1, 1)} = x^{-e_1 + e_2}$, and the last row and last column correspond to $x^{(0, -1)} = x^{-e_2}$; all empty entries are 0 (the diagonal zero entries in $\partial_2$ and $R_2$ are kept to facilitate visual alignment of matrix elements). Note that basis elements are ordered by $C_1^+$, $C_1^-$, $C_1^0$, $x^{e_2}$, $x^{-e_2}$.

Then we construct $V_D$ and $R_D$ inductively. Let
\begin{align*}
L_D &= \{x^\xi \in C_D: x^\xi = \text{low}(x^\zeta) \text{ in relation to $R_D$ for some } x^\zeta\}
\\\mathbf{1}_D &= \sum_{1 \leq i \leq D} e_i
\\ B_D &= C_D - L_D - \{x^{\mathbf{1}_D}, x^{-e_1}\}.
\end{align*}
Note that assuming $R_{D-1}$ is reduced, then $\text{low}$ is a bijective function between $L_{D-1}$ and $\{x^\zeta: R_{D-1} x^\zeta \neq 0\}$. Under this assumption, we define $V_D$ and $R_D$ as follows.
\begin{align}\label{eqn:construct_factorization}
R_D: &
\begin{blockarray}{ccccccccc}
 & C_{D-1}^+ & C_{D-1}^- & x^{\mathbf{1}_{D-1}} & L_{D-1} & B_{D-1} & x^{-e_1} & x^{e_D} & x^{-e_D} \\
\begin{block}{c[cccccccc]}
  C_{D-1}^+    & I^+_{D-1} R_{D-1}(I^+_{D-1})^{-1} &                           & I^+_{D-1} & & I^+_{D-1}          & & u^+_{D-1}   &          \\
  C_{D-1}^-    &                           & I^-_{D-1} R_{D-1}(I^-_{D-1})^{-1} & I^-_{D-1} & & I^-_{D-1}          & &         & u^-_{D-1}    \\
  C_{D-1}^0    &                           &                           &       & & \partial_{D-1} & &         &          \\
  x^{e_D}  &                           &                           &       & &                & &         &          \\
  x^{-e_D} &                           &                           &       & &                & &         &          \\
  \end{block}
\end{blockarray},
\notag
\\
V_D: &
\begin{blockarray}{ccccccccc}
  & C_{D-1}^+ & C_{D-1}^- & x^{\mathbf{1}_{D-1}} & L_{D-1} & B_{D-1} & x^{-e_1} & x^{e_D} & x^{-e_D} \\
\begin{block}{c[cccccccc]}
  C_{D-1}^+    & I^+_D V_{D-1}(I^+_D)^{-1} &                           &   & - I^+_D V_{D-1} \text{low}^{-1} &   &         &   &  \\
  C_{D-1}^-    &                           & I^-_D V_{D-1}(I^-_D)^{-1} &   & - I^-_D V_{D-1} \text{low}^{-1} &   &         &   &  \\
  C_{D-1}^0    &                           &                           & I & R_{D-1} \text{low}^{-1}         & I & V_{D-1} &   &  \\
  x^{e_D}  &                           &                           &   &                                 &   & 1       & 1 &  \\
  x^{-e_D} &                           &                           &   &                                 &   & 1       &   & 1\\
\end{block}
\end{blockarray},
\end{align} 
where $\text{low}^{-1}: \text{span }L_{D-1} \to \text{span }C_{D-1}^0$ in the third column of $V_D$ is defined by
$$\text{low}^{-1}(x^\xi) = x^\zeta$$
whenever $x^\xi = \text{low}(R_{D-1} x^\zeta)$, and the $I$'s in $V_D$ without superscripts or subscripts refer to inclusion, e.g. the $I$ in the fifth column is the inclusion of $B_{D-1}$ to $C_D^0$. Note that some blocks in the above matrices do not necessarily have the supposed dimensions, as we have suppressed some inclusion maps, e.g. the domain of $\partial_D$ in the fourth column of $R_D$ is $B_{D-1} \subseteq C^0_D$ as opposed to the whole $C^0_D$.

We are now ready to prove that $\partial_D$ is column reduced via $R_D = \partial_D V_D$ in \cref{prop:column_reduce_boundary}. Most facts in the proposition are mundane and straight-forward, but the mutual dependence of these facts necessitates a single unified induction argument. For instance, $V_D$ as of now is ill-defined, because we have not established that $R_D$ is reduced. Therefore, all intervening steps depend on the induction hypothesis that $R_{D-1}$ is reduced, and hence they need to be established in the same induction argument.

\begin{proposition} \label{prop:column_reduce_boundary}
Suppose $D \geq 2$. 
\begin{enumerate}
\item $V_D x^{-e_1} = u_D$, $R_D x^{-e_1} = 0$, and $R_D x^{\mathbf{1}_D} = 0$.
\item $R_D = \partial_D V_D$.
\item $V_D$ restricted on $B_D$ is identity, and hence $\partial_D$ and $R_D$ coincide on $B_D$. Further, $\partial_D x^\xi = R_D x^\xi \neq 0$ for $x^\xi \in B_D$.
\item \begin{align*}
I^-_{D-1} x^{\mathbf{1}_{D-1}}&= \text{low}(R_D x^{\mathbf{1}_{D-1}})\\
I^\pm_{D-1} x^{-e_1} &= \text{low}(R_D x^{\pm e_D}) \qquad \text{(the $-e_1$ is \emph{not} a typo)}
\end{align*}
\item $L_D = I^+_{D-1}(L_{D-1}) \cup I^-_{D-1}(L_{D-1}) \cup L_{D-1} \cup \{I^-_{D-1} x^{\mathbf{1}_{D-1}}, I^+_{D-1} x^{-e_1}, I^-_{D-1} x^{-e_1}\}$.
\item $x^{-e_1}$ and $x^{\mathbf{1}_D}$ are not lowest generators of any $R_D x^\xi \in C_D$.
\item $R_D$ is reduced, and hence $V_D$ is well-defined.
\item $\partial_D$ is column reduced via $R_D = \partial_D V_D$.
\item Let \begin{equation*}
P_D = \{(x^\xi, x^\zeta) \in C_D \times C_D : x^\xi = \text{low}(R_D x^\zeta)\},
\end{equation*}
and we extend the definitions of $I^\pm_D$ to $C_D \times C_D$ by entrywise action.
Then 
\begin{align}\label{eqn:finite_PD_recursion_critical_points}
P_{D+1} &= I^+_{D} P_D \cup I^-_{D} P_D \cup P_D  \notag
\\& \qquad \cup \{(I^-_D x^{\mathbf{1}_D}, x^{\mathbf{1}_D}), (I^+_{D}x^{-e_1}, x^{e_{D+1}}), (I^-_{D}x^{-e_1}, x^{-e_{D+1}})\}.
\end{align}
\end{enumerate}
\end{proposition}

\begin{proof}
We prove the proposition inductively. The case for $D = 2$ can be directly verified.

For the first point, $V_D x^{-e_1} = u_D$ follows from the inductive hypothesis for this point itself, $R_D x^{-e_1} = 0$ follows directly from the structure of $R_D$, and $R_D x^{\mathbf{1}_D} = 0$ again follows inductively.

The proof for the second point is a block-by-block checking using \cref{eqn:construct_factorization}. Most checking follows from the inductive hypothesis that $R_{D-1} = \partial_{D-1} V_{D-1}$. The only nontrivial blocks are in the columns for $x^{-e_1}$ and the rows for $C^0_{D-1}$. For the column of $x^{-e_1}$, we need to show
$$
\begin{bmatrix}
I^+_{D-1} (V_{D-1} x^{-e_1} + u_{D-1})\\
I^-_{D-1} (V_{D-1} x^{-e_1} + u_{D-1})\\
\partial_{D-1} V_{D-1} x^{-e_1}
\end{bmatrix} = 0.
$$
Note that the last block is $R_{D-1} x^{-e_1}$, by the inductive hypothesis. Then all three blocks vanish because of the inductive hypothesis for the first point of the proposition, as well as $\mathbb{Z}/2\mathbb{Z}$ arithmetic.

We will establish the third point after establishing the identity of $L_D$ in terms of $L_{D-1}$ and other subsets (the fifth point of the proposition).


For the fourth point, $R_D x^{\mathbf{1}_{D-1}} = I_{D-1}^+ x^{\mathbf{1}_{D-1}} + I_{D-1}^- x^{\mathbf{1}_{D-1}}$. By our convention on the ordering, $\text{low}(R_D x^{\mathbf{1}_{D-1}}) = I_{D-1}^- x^{\mathbf{1}_{D-1}}$.
$$R_D x^{\pm e_D} = I_{D-1}^\pm u_{D-1}
= I_{D-1}^\pm (\sum_{1 \leq i \leq D-1} x^{e_i} + x^{-e_i})
= \sum_{1 \leq i \leq D-1} x^{e_i \pm e_D} + x^{-e_i \pm e_D}.$$
By our convention, $\pm e_i \pm e_D \prec -e_1 \pm e_D$, and hence $\text{low}(R_D x^{\pm e_D}) = x^{-e_1 \pm e_D} = I^\pm_{D-1} x^{-e_1}$.

For the fifth point, the recursive definition of $R_D$ and the previous point together suggest that $L_D$ contains
$$I^+_{D-1}(L_{D-1}) \cup I^-_{D-1}(L_{D-1}) \cup \{I_{D-1}^- x^{\mathbf{1}_{D-1}}, I^+_{D-1} x^{-e_1}, I^-_{D-1} x^{-e_1}\}.$$
It remains to consider $C^0_{D-1}$. The inductive hypothesis of the third point implies the column of $R_D$ for $C^0_{D-1}$ is the same as
\begin{equation}\label{eqn:RD_BD-1}
\begin{bmatrix}
I^+_{D-1}\\
I^-_{D-1}\\
R_{D-1}
\end{bmatrix}.
\end{equation}
(Even though we have not established the third point for the case $D$, we are allowed to use the third point for the case $D-1$.)
The same point also implies the lowest block $R_{D-1}: B_{D-1} \to C^0_{D-1}$ has no zero column. This block also contains all the nonzero column of $R_{D-1}$, because the other columns vanish (by \cref{lem:R_of_lowest_one_must_vanish} for $L_{D-1}$ and by the first point of the proposition for $x^{-e_1}, x^{\mathbf{1}_D}$), and hence all lowest generators from the column $C^0_{D-1}$ of $R_D$ come from $B_{D-1}$. The fifth point then follows if these lowest generators dominate those in the first two blocks in \cref{eqn:RD_BD-1}. To show this, the third point of the proposition implies the lowest block in \cref{eqn:RD_BD-1} coincides with $\partial_{D-1}$. Therefore, it suffices to show the lowest generators of $\partial_{D-1} x^\xi$ are in $C^0_{D-1}$ for $x^\xi \in B_{D-1}$. Let $\xi_i$ be the first zero entry of $\xi$. Since $x^\xi \in B_{D-1} \subseteq C^0_{D-1} - \{\mathbf{1}_{D-1}\}$, $i < D$. Then by \cref{lem:low_boundary}, $\text{low}(\partial_{D-1} x^\xi) = x^{-e_i + \xi}$. Since $f_{k,D}(x^{-e_i + \xi}) > f_{k,D}(x^{\pm e_D + \xi})$, the lowest generator indeed lies in $C^0_{D-1}$.

We now establish the third point. The fifth point implies
$$B_D = I^+_{D-1}(B_{D-1}) \cup I^-_{D-1}(B_{D-1}) \cup B_{D-1} \cup \{x^{\mathbf{1}_{D-1}}, x^{e_D}, x^{-e_D}\}.$$
The third point then follows from direct inspection.

We will establish the sixth point after the seventh. For the seventh point, to see $R_D$ is reduced, the seventh point implies the lowest generators of $x^{\mathbf{1}_{D-1}}, x^{e_D}, x^{-e_D}$ are elements of $C^\pm_{D-1}$ that are not lowest generators of other elements in $C^\pm_D$. The induction hypothesis implies the rest of the columns do not have common lowest generators. The seventh point then follows.

For the sixth point, $x^{-e_1}$ cannot be a lowest generator because it is the last element of $C_D$ and the fact that $\partial_D$ by construction is upper triangular. If $x^{\mathbf{1}_D} = \text{low}(R_D x^\zeta)$ is a lowest generator, then by the structure of $R_D$ and our knowledge of the lowest generators in the arguments above, $x^\zeta \in C^+_{D-1}$, and hence $x^{\mathbf{1}_{D-1}} = \text{low}(R_{D-1} (I^+_{D-1})^{-1}x^\zeta)$, contradictory to the inductive hypothesis.

For the eighth point, it remains to show $V_D$ is unit upper triangular. Since $V_D$ is identity on $B_D$, it suffices to consider the column for $L_{D-1}$. It can be readily shown that $\text{low}(R_{D-1} \text{low}^{-1} x^\xi) = x^\xi$, and hence it suffices to show that whenever $\xi = \text{low}(R_{D-1} x^\zeta)$, generators of $I^\pm_{D-1} x^\zeta$ all precede $x^\xi$ itself. Since $\partial_{D-1}$ and $R_{D-1}$ coincide on $B_{D-1}$, we have $x^\xi = \text{low}(\partial_{D-1} \zeta) = x^{-e_i + \zeta}$, where $\zeta_i$ is the first zero entry of $\zeta$. On the other hand, $I^\pm_{D-1} x^\zeta = x^{\pm e_D + \zeta},$ which has strictly smaller function value than $x^{-e_i + \zeta}$. Upper triangular-ness then follows.

The last point merely summarizes the observations about lowest generators in the argument above.
\end{proof}

\section{Proof of \cref{thm:persistence_diagram_sublevel_even}}
\label{sec:proof_sublevel_even}
The case for $D = 1$ is trivial. The case for $D = 2$ can be directly verified. We only consider the case for $D \geq 3$.

For the first claim, by \cref{thm:persistence_diagram_from_reduced_matrix}, it suffices to find $x^\xi$ such that $R_D x^\xi = 0$, and $x^\xi$ is not the lowest generator of any $Rx^\zeta$. The first and the sixth points of \cref{prop:column_reduce_boundary} then implies $x^{-e_1}$ and $x^{\mathbf{1}_D}$ are such critical points. The third point of \cref{prop:column_reduce_boundary} implies they are the only such generators, because $C_D = \{x^{-e_1}, x^{\mathbf{1}_D}\} \cup L_{D} \cup B_{D}$, where $L_{D}$ by construction consists of lowest generators, and $R_D$ does not vanish on $B_D$.

We prove the second claim by induction. The case for $D = 2$ follows by inspecting $R_2$ directly. By \cref{thm:persistence_diagram_from_reduced_matrix}, \cref{eqn:finite_PD_recursion_critical_points} in the last point of \cref{prop:column_reduce_boundary} defines a recursion of the finite parts of the union of persistence diagrams of $f_{k,D}$ across dimensions. 
Let $\mathcal{D}^\text{finite}_d$ be the union of finite parts of the union of persistence diagrams of $f_{k,d}$ on $S^{d-1}$ across homological dimensions. (We switch to denote by ambient dimension by the lowercase $d$ to help distinguish it from the notation for persistence diagram.) Note that the coordinates of points in $\mathcal{D}^\text{finite}_d$ is of the form $\nu_A$ for some $A \subseteq \{1, ..., d\}$, because all critical values take this form (Cf. \cref{lem:critical_values_and_ordering}).
Then
\begin{align}\label{eqn:finite_PD_recursion_critical_values}
\mathcal{D}^\text{finite}_d &= 2 \times \Phi_d (\mathcal{D}^\text{finite}_{d-1}) \cup \mathcal{D}^\text{finite}_{d-1} \notag \\& \qquad \cup \{(\nu_{\{1, ..., d\}}, \nu_{\{1, ..., d-1\}})\} \cup 2 \times \{(\nu_{\{1, d\}}, \nu_{\{d\}})\},
\end{align}
where
\begin{itemize}
\item $2 \times E$ denotes the set with the same elements but the multiplicity of each element in $2 \times E$ is double the original multiplicity, e.g. $2 \times \{a, b, b\} = \{a, a, b, b, b, b\}$, and
\item $\Phi_d(E) = \{\phi_d(\nu): \nu \in E\}$, where $\phi_d(\nu_A, \nu_B) = (\nu_{A \cup \{d\}}, \nu_{B \cup \{d\}})$.
\end{itemize}

By the induction hypothesis, $\mathcal{D}^\text{finite}_{d-1}$ consists of points of the form $(\nu_E, \nu_{E - \{i\}})$ with multiplicity $2^{|E| - i}$, where $E$ ranges over all subsets of $\{1, ..., d-1\}$ containing 1 and has at least 2 elements (because we have taken union across homological dimensions).
It suffices to show the recursion implies $\mathcal{D}^\text{finite}_d$ has the same form, because the cardinality of $E$ in the coordinates of points in $\mathcal{D}^\text{finite}_d$ automatically keeps track of the homological dimension. Indeed, \cref{eqn:finite_PD_recursion_critical_values} does imply that $\mathcal{D}^\text{finite}_d$ has the desired form, as $2 \times \Phi_d (\mathcal{D}^\text{finite}_{d-1})$ contributes to the $(\nu_E, \nu_{E - \{i\}})$'s such that $d \in E$, $|E| \geq 3$ and $i \neq d$; $\mathcal{D}^\text{finite}_{d-1}$ contributes to the $(\nu_E, \nu_{E - \{i\}})$'s that $d \notin E$; $\{(\nu_{\{1, ..., d\}}, \nu_{\{1, ..., d-1\}})\}$ contributes to the case where $d \in E$, $|E| \geq 3$ and $d = i$ (since it is required that $i \leq |E|$, this case means $E$ must be $\{1, ..., d\}$); and $2 \times \{(\nu_{\{1, d\}}, \nu_{\{d\}})\}$ contribute to the case where $|E| = 2$, and $d \in E$. The second claim then follows.

\section{Proof of \cref{cor:persistence_diagram_superlevel_even}}
\label{sec:proof_cor_persistence_diagram_superlevel_even}

First, we state the version of Alexander duality we need:
\begin{theorem}[Alexander Duality]\label{thm:alexander_duality}
Let $K \subseteq L \subseteq S^{D-1}$. Suppose $K$ and $L$ are compact, locally contractible, nonempty and proper (proper means not the whole $S^{D-1}$). Then there exist horizontal isomorphisms such that the following diagram commutes:
$$
\begin{tikzcd}
\tilde H_q (S^{D-1} - L; \mathbb{Z}/2\mathbb{Z}) \arrow[d] \arrow[r] & \tilde H^{D-2-q}(L; \mathbb{Z}/2\mathbb{Z}) \arrow[d] \\
\tilde H_q (S^{D-1} - K; \mathbb{Z}/2\mathbb{Z}) \arrow[r]           & \tilde H^{D-2-q}(K; \mathbb{Z}/2\mathbb{Z}),          
\end{tikzcd}
$$
where vertical maps are induced by inclusions and $\tilde H$ denotes reduced homology (defined in \cref{def:reduced_homology}). 
\end{theorem}

\begin{remark}
A space $X$ is said to be ``locally contractible" if for each $x \in X$ and each neighborhood $U$ of $x$ there exists a neighborhood $V \subseteq U$ of $x$ such that the inclusion map $V \to U$ is homotopic to the constant map $y \mapsto x$ for $y \in V$ (Cf. Theorem A.7 and p.3 -- 4 of \citep{hatcher02_algtopo}). Naturality of Alexander duality was established in Theorem 72.4 of \citep{munkres84algtopo}.
\end{remark}

\begin{theorem}[Proposition 2.3 of Morozov's duality theorem]\label{thm:persistent_duality}
The persistent homology and the persistent cohomology of a filtered chain complex (Cf. \cref{def:filtered_chain_complex}) have the same persistence diagram at all dimensions.
\end{theorem}

\begin{lemma}\label{lem:strict_sublevel_homologous_to_sublevel}
Let $f$ be a Morse function and $c$ be a regular value (i.e. not a critical value). Then the inclusion map $f^{-1}(c, \infty) \to f^{-1}[c, \infty)$ induces an isomorphism on homology.
\end{lemma}
\begin{proof}[Sketch of Proof]
Pick a $c'$ such that $c' > c$ and there is no critical value in $[c, c']$.
By Lemma 3.1 of \cite{milnor63_morseTheory} (applied to $-f$) $f^{-1}[c, \infty)$ deformation retracts to $f^{-1}[c', \infty)$ via the gradient ascent flow. By exactly the same argument in the proof, $f^{-1}(c, \infty)$ also deformation retracts to $f^{-1}[c', \infty)$. Thus the inclusions of $f^{-1}[c', \infty)$ into both of them induce isomorphisms $\varphi_{[c, \infty)}, \varphi_{(c, \infty)}$. It follows that the homology map induced by inclusion map $f^{-1}(c, \infty) \to f^{-1}[c, \infty)$, which is $\varphi_{[c, \infty)}^{-1}\varphi_{(c, \infty)}$, is an isomorphism.
\end{proof}

\begin{lemma}\label{lem:persistence_diagram_reduced_homology}
Let $f$ be a Morse function. Let $\mathcal{D}_q$ and $\tilde{\mathcal{D}}_q$ be the persistence diagrams of the persistent homology and persistent \emph{reduced} homology of the sublevel filtration of $f$ at dimension $q$.
Then $\mathcal{D}_q = \tilde{\mathcal{D}}_q$ for $q > 0$, and $\mathcal{D}_0 = \tilde{\mathcal{D}}_0 \cup \{(\min f, \infty)\}$. 
\end{lemma}

We defer the proof of this lemma to \cref{sec:homological_algebra}.

\begin{lemma}
Let $f$ be a Morse function on a closed manifold. Suppose $t$ is a regular value of $f$ (i.e. not critical value) and $\min f < t < \max f$. Then $f^{-1}((-\infty, t])$ is compact, locally contractible, nonempty and proper.
\end{lemma}

\begin{proof}
Nonemptiness and properness are clear. Compactness follows from the compactness of closed sets in a closed (and hence compact) manifold. Under our assumption, points in $f^{-1}((-\infty, t])$ have local bases consisting of either open balls or half-balls, both of which are contractible, and hence their identity maps are homotopic to constant maps.
\end{proof}

\begin{proof}[Proof of \cref{cor:persistence_diagram_superlevel_even}]
The case for $D = 1$ is trivial. Suppose $D \geq 2$. Let $f = f_{k,D}$, and $x_0 = e_1$, which is a global maximum of $f$. Throughout this proof, the coefficients of all homology groups are in $\mathbb{F} = \mathbb{Z}/2\mathbb{Z}$, and the base point for reduced (co-) homology is $x_0$.

\begin{table}
\centering
\caption{Homology $H_q(\{x: f(x) \geq t\}; \mathbb{F})$ of the superlevel filtration of $f = f_{k,D}$, where $\mathbb{F} = \mathbb{Z}/2\mathbb{Z}$, $D \geq 2$, and $k$ is even. AD means the homology can be computed by applying Alexander duality to \cref{thm:persistence_diagram_sublevel_even}.}
\label{tab:homology_superlevel_even}
\begin{tabular}{|c|c c c c c|}
\hline
 $q$ & $t < \min f$  & $t = \min f$ & $\min f < t < \max f$ & $t = \max f$ & $t > \max f$ 
\\
\hline
$\{0\}$ & $\mathbb{F}$ & $\mathbb{F}$ & AD & $\mathbb{F}^2$ & 0
\\ 
$[1, D-2]$ & 0 & 0 & AD & 0 & 0
\\
$\{D-1\}$ & $\mathbb{F}$ & $\mathbb{F}$ & 0 & 0 & 0
\\
$[D, \infty)$ & 0 & 0 & 0 & 0 & 0
\\
\hline
\end{tabular}
\end{table}

The homology groups of the superlevel sets of $f$ at different dimensions are summarized in \cref{tab:homology_superlevel_even}. Assertions for all columns except the middle one (for $\min f < t < \max f$) follows from the fact that the superlevel sets on those columns are either empty ($t \geq \max f$), $\{e_1, -e_1\}$ ($t = \max f$), or $S^{D-1}$ ($t \leq \min f$). For dimension $D - 1$ or higher, viewing $S^{D-1} - \{x_0\}$ as $\mathbb{R}^{D-1}$, Corollary 3.46 of \citep{hatcher02_algtopo} (and the universal coefficient theorem, Theorem 3A.3 of \citep{hatcher02_algtopo}) implies $H_q(f^{-1}[t, \infty)) = 0$ for $\min f < t < \max f$. The corollary for dimensions $q \geq D-1$ then follows.

It remains to use Alexander duality to deduce the homology at dimensions $0 \leq q \leq D-2$ for the range $\min f < t < \max f$, where the persistent \emph{reduced} homology is defined.

Alexander duality (\cref{thm:alexander_duality}), along with \cref{lem:strict_sublevel_homologous_to_sublevel}, implies that persistence module ${\tilde M}_q^\text{superlevel homology}$ of the persistent reduced homology of the superlevel filtration of $f$ at dimension $q$ is isomorphic to the persistence module ${\tilde M}_{D-2-q}^\text{sublevel cohomology}$ of the persistent reduced cohomology of the sublevel filtration of $f$ at dimension $D-2-q$, whose persistence diagram $\tilde{\mathcal{D}}_{D-2-q}^\text{sublevel cohomology}$ in turn is equal to the persistence diagram $\tilde{\mathcal{D}}_{D-2-q}^\text{sublevel homology}$
 of the persistent reduced homology of the sublevel filtration of $f$ at dimension $D-2-q$ by \cref{thm:persistent_duality}.
 
It remains to compare the reduced homology and homology of the superlevel filtration. For $1 \leq q \leq D-2$, since the homology groups in the first and the last columns of \cref{tab:homology_superlevel_even} for the second row vanish, by \cref{lem:persistence_diagram_reduced_homology}, the persistence diagram of the reduced homology and homology coincide. The corollary for these dimensions then follows.

For dimension $q = 0$, the Alexander duality argument above shows that the persistence module\\${\tilde M}_0^{\text{superlevel homology on }(\min f, \max f)}$ of the persistent reduced homology of the superlevel filtration of $f$ at dimension $0$ on the range $(\min f, \max f)$ is the direct sum of interval modules whose intervals take the form $(\nu_{\{1, i\}}, \nu_{\{i\}})$ for $i \geq 2$ or $(\nu_{\{1, 2\}}, \nu_{1})$ (modulo the inclusion of endpoints), and hence the persistence module $M_0^{\text{superlevel homology on } (\min f, \max f)}$ is the direct sum of interval modules of such forms and the $(\min f, \max f)$-interval module, generated by the connected component of $x_0$. We need to check the relationship between the $(\min f, \max f)$-interval module and the generators of the $\mathbb{F}$ in the first two columns for the row $q = 0$ in \cref{tab:homology_superlevel_even}. Indeed, in the persistence module $M_0^{\text{superlevel homology on } \mathbb{R}}$, this $(\min f, \max f)$-interval module extends to $-\infty$, because the $\mathbb{F}$ in the first two columns for the row $q = 0$ are also generated by the connected component of $x_0$. The corollary then follows.
\end{proof}

\section{The Case for Odd $k$}
\label{sec:odd_k}

Let ${\tilde C}_D$ be the concatenated complex of the reduced Morse complex (defined in \cref{def:reduced_homology}) of $f_{k,D}$ generated by critical points $x^\xi$ such that $\xi_i \leq 0$ for every $1 \leq i \leq D$, with concatenated differential ${\tilde \partial}_D: {\tilde C}_D \to {\tilde C}_D$. ${\tilde C}_D$ is well-defined because \cref{prop:morse_differential_tensor} guarantees that if $\xi_i \leq 0$ for every $1 \leq i \leq D$, then $\tilde \partial x^\xi = \partial x^\xi \in {\tilde C}_D$. The persistent homology for the full complex will be obtained by appealing to Alexander duality. Let
\begin{align*}
{\tilde C}^0_{D-1} &= \{x^\xi : \xi_i \leq 0, \xi_D = 0\} \cup \{x^0\} \\
{\tilde C}^-_{D-1} &= \{x^\xi : \xi_i \leq 0, \xi_D = -1\},
\end{align*}
where $x^0$ denotes the generator of ${\tilde C}_{-1}$. For example,
\begin{equation*}
{\tilde \partial}_1:
\begin{blockarray}{ccc}
  & x^0 & x^{-e_1} \\
\begin{block}{c[cc]}
  x^0      &   & 1 \\
  x^{-e_1} &   &   \\
\end{block}
\end{blockarray}.
\end{equation*}
Note that with the notation $x^0$, if we treat $0 = (0, ..., 0) \in \{-1, 0, 1\}^D$, \cref{eqn:morse_differential_odd} in \cref{prop:morse_differential_tensor} simplifies to
\begin{equation}\label{eqn:morse_reduced_differential_odd}
\tilde \partial x^\xi = \sum_{\substack{1 \leq i \leq D \\ \xi_i = -1}} x^{e_i + \xi}.
\end{equation}

Define $I_{D-1}: {\tilde C}^-_{D-1} \to {\tilde C}^0_{D-1}$ by $$I_{D-1} x^\xi = x^{\xi + e_D}
$$
Then $I_{D-1}$ is bijective.

We also define $J: {\tilde C}^-_{D-1} \to {\tilde C}^-_{D-1}$ by
$$J x^\xi = 
\begin{cases}
x^{\xi - e_1} & \text{ if } \xi_1 = 0 \\
0 & \text{ otherwise}.
\end{cases}$$
We do not subscript $J$ by $D-1$ because the above formula does not depend on $D$.

\begin{proposition}\label{prop:morse_differential_block_structure_odd}
We have the following recurrence for $D \geq 2$:
$${\tilde \partial}_D: 
\begin{blockarray}{ccc}
 & C_{D-1}^0 & C_{D-1}^- \\
\begin{block}{c[cc]}
  C_{D-1}^0    & {\tilde \partial}_{D-1} & I_{D-1}\\
  C_{D-1}^-    &   & I_{D-1}^{-1}{\tilde \partial}_{D-1}I_{D-1}\\
  \end{block}
\end{blockarray}.$$
\end{proposition}

\begin{proof}
Direct application of \cref{prop:morse_differential_tensor}:
$$\tilde \partial x^\xi  = \sum_{\substack{1 \leq i \leq D \\ \xi_i = -1}} x^{e_i + \xi} = \sum_{\substack{1 \leq i \leq D-1 \\ \xi_i = -1}} x^{e_i + \xi} + x^{e_D + \xi},$$
where the sum corresponds to the lower-right block and the second term corresponds to the upper-right block. The left blocks are obvious.
\end{proof}

We column reduce $\tilde \partial_D$ inductively. We order $\tilde C_D$ with $\prec$ with the same definition as in \cref{def:order_on_chain_group}, except that the second condition $1 \prec -1 \prec 0$ degenerates to $-1 \prec 0$. Again, it can be directly verified that $I_D$ is monotone with respect to $\prec$ and 
\begin{equation}
I_{D-1} x^\xi \prec x^\xi.
\end{equation}
We define $V_D$ and $R_D$ as follows. Let $V_1$ be the identity and $R_1 = \tilde \partial_1$. Then, inductively,
\begin{align}
V_D: &
\begin{blockarray}{ccc}
  & {\tilde C}^0_{D-1} & {\tilde C}^-_{D-1} \\
\begin{block}{c[cc]}
  {\tilde C}^0_{D-1}    & V_{D-1} & J I_{D-1} \\
  {\tilde C}^-_{D-1}     &         & I_{D-1}^{-1} V_{D-1} I_{D-1} \\
  \end{block}
\end{blockarray} \notag
\\
R_D: &
\begin{blockarray}{ccc}
  & {\tilde C}^0_{D-1} & {\tilde C}^-_{D-1} \\
\begin{block}{c[cc]}
  {\tilde C}^0_{D-1}    & R_{D-1} & (V_{D-1} + \tilde \partial_{D-1} J)I_{D-1} \\
  {\tilde C}^-_{D-1}     &         & I_{D-1}^{-1}R_{D-1}I_{D-1} \\
  \end{block}
\end{blockarray}
\end{align}

\begin{proposition}\label{prop:column_reduction_odd}
$\tilde \partial_D$ is column reduced via $R_D = \tilde \partial_D V_D$ with respect to $\prec$. Further, if $R_D x^\xi = 0$, then $\xi_1 = 0$ and $V_D x^\xi = \tilde \partial_D Jx^\xi$; if $R_D x^\xi \neq 0$, then $\xi_1 \neq 0$, $V_D x^\xi = x^\xi$ and hence $R_D x^\xi = \tilde \partial_D x^\xi$, and its lowest generator is $x^{\xi + e_1}$.
\end{proposition}

\begin{proof}
The case for $D = 1$ is trivial. For the inductive step, $R_D = \tilde \partial_D V_D$ holds inductively by \cref{prop:morse_differential_block_structure_odd} and by construction.

To check that $V_D$ is unit-upper-triangular, since $I_{D-1}$ is monotone, it suffices to show that
$$JI_{D-1}x^\xi \prec x^\xi.$$
Note that simply $JI_{D-1}x^\xi = x^{\xi + e_D - e_1}$. Then the desired relation follows directly from computation and the facts that $\lambda_1 \geq \lambda_D$ and $-1 \prec 0$.

Before showing $R_D$ is reduced, we first show the claim in the proposition about zero and nonzero columns of $R_D$. By the inductive hypothesis, it suffices to consider $x^\xi \in C^-_{D-1}$.

Suppose $R_D x^\xi = 0$. Then
\begin{equation}\label{eqn:zero_column_R_odd_first_block}
(V_{D-1}I_{D-1} + \tilde \partial_{D-1} JI_{D-1}) x^\xi = 0
\end{equation}
and $R_{D-1} (I_{D-1} x^\xi) = 0$. The latter implies $\xi_1 = 0$ via the inductive hypothesis. We verify that $V_D x^\xi = \tilde \partial_D Jx^\xi$ as follows.
\begin{align*}
V_D x^\xi
&= JI_{D-1} x^\xi + I_{D-1}^{-1} V_{D-1} I_{D-1} x^\xi
\\&= JI_{D-1} x^\xi + I_{D-1}^{-1} \tilde \partial_{D-1} JI_{D-1} x^\xi & \text{by \cref{eqn:zero_column_R_odd_first_block}, coeff in $\mathbb{Z}/2\mathbb{Z}$}
\\&= x^{\xi + e_D - e_1} + I_{D-1}^{-1} \left(x^{\xi + e_D} + \sum_{\substack{2 \leq i \leq D-1 \\ \xi_i \neq 0}} x^{\xi + e_i - e_1 + e_D} \right) & \text{by \cref{eqn:morse_reduced_differential_odd} and } \xi_1 = 0
\\&= x^{\xi + e_D - e_1} + x^{\xi} + \sum_{\substack{2 \leq i \leq D-1\\\xi_i \neq 0}} x^{\xi + e_i - e_1}
\\&= Jx^{\xi + e_D} + x^{\xi} + \sum_{\substack{2 \leq i \leq D-1\\\xi_i \neq 0}} Jx^{\xi + e_i}
\\&= \tilde \partial_D Jx^\xi. & \text{by \cref{eqn:morse_reduced_differential_odd}}
\end{align*}

Suppose $R_D x^\xi \neq 0$. Then at least one of $(V_{D-1} + \tilde \partial_{D-1} J) I_{D-1} x^\xi$ and $I_{D-1}^{-1} R_{D-1} I_{D-1} x^\xi$ is nonzero. We claim that the latter must be nonzero, and hence the inductive hypothesis implies the first coordinate of $I_{D-1} x^\xi = x^{\xi + e_D}$, which is simply $\xi_1$, is nonzero. Suppose, for a contradiction, that $I_{D-1}^{-1} R_{D-1} I_{D-1} x^\xi$ is zero. The inductive hypothesis implies $(V_{D-1} + \tilde \partial_D J)I_{D-1} x^\xi = 0$. Then both components of $R_D x^\xi$ are zero, contradictory to our assumption that $R_D x^\xi \neq 0$.

To show $V_D x^\xi = x^\xi$, since $\xi_1 \neq 0$, the first component of $V_D x^\xi$, namely $JI_{D-1} x^\xi$, vanishes by the definition of $J$. Therefore, $V_Dx^\xi$ is simply the second component: $I_{D-1}^{-1} R_{D-1} I_{D-1} x^\xi$, which, in the preceding paragraph, has been shown to be nonzero. The induction hypothesis then implies $V_{D-1} I_{D-1}x^\xi = I_{D-1}x^\xi$, and hence $V_Dx^\xi = I_{D-1}^{-1} R_{D-1} I_{D-1} x^\xi = x^\xi$.

Then $R_D x^\xi = \tilde \partial_D V_D x^\xi = \tilde \partial_D x^\xi$. Its lowest generator takes the form $x^{\xi + e_i}$ for some $i$. Considering their critical values shows that $i$ must be the smallest among $j$'s such that $\xi_j \neq 0$. Since $\xi_1 \neq 0$, the lowest generator is $x^{\xi + e_1}$.

Finally, we check that $R_D$ is reduced. The lowest generators of nonzero $R_D x^\xi$ with $x^\xi \in \tilde C^0_{D-1}$ are by construction in $\tilde C^0_{D-1}$, and they are unique among themselves because $R_{D-1}$ is reduced. For lowest generators of nonzero $R_D x^\xi$ with $x^\xi \in \tilde C^-_{D-1}$, by the last claim of the proposition, which has been established in the preceding paragraphs, they are simply $x^{\xi + e_1} \in \tilde C^-_{D-1}$, which must be distinct from those in $\tilde C^0_{D-1}$. They must also be distinct among themselves, because $\xi \mapsto \xi + e_1$ is injective.

\end{proof}

\begin{proof}[Proof of \cref{thm:persistence_diagram_sublevel_odd}]
By \cref{lem:persistence_diagram_reduced_homology}, it suffices to consider reduced homology. \cref{prop:column_reduction_odd} shows that the persistence diagram of the persistent reduced homology of $f_{k,D}$ over the range $[\min f_{k,D}, 0]$ consists of points $(f_{k,D}(x^\xi), f_{k,D}(x^{\xi - e_1}))$ with multiplicity 1, where $\xi_1 \neq 0$ and the homological dimension is determined by $\{i: \xi_i \neq 0\}$. These points give rise to the points indexed by $E^-$ in \cref{thm:persistence_diagram_sublevel_odd}. Note that all death times are strictly less than 0.

By Alexander duality and duality between homology and cohomology (\cref{thm:alexander_duality,thm:persistent_duality}), we have
\begin{align*}
H_q(\{x: f(x) \leq t\})
& \cong H_q(\{x: f(x) < t\}) & \text{(\cref{lem:strict_sublevel_homologous_to_sublevel})}
\\& \cong H_q(S^{D-1} - \{x: f(x) \geq t\})
\\& \cong H^{D-2-q}(\{x: f(x) \geq t\}) & \text{(Alexander duality, i.e. \cref{thm:alexander_duality})}
\\& \cong H_{D-2-q}(\{x: f(x) \geq t\}) & \text{(duality, i.e. \cref{thm:persistent_duality})}
\\& \cong H_{D-2-q}(\{y: f(-y) \geq t\}) & \text{(The homeomorphism $x \mapsto -x$}
\\&& \qquad \qquad \text{induces a homological isomorphism)}
\\& \cong H_{D-2-q}(\{y: f(y) \leq -t\})
\end{align*}

Therefore, the persistence diagram at a dimension over the range $[0, \max f_{k,D}]$ is the persistence diagram at the dual dimension over the range $[\min f_{k,D}, 0]$ with a sign flip and birth-death flip. This gives rise to points indexed by $E^+$ in \cref{thm:persistence_diagram_sublevel_odd}.
\end{proof}

\section{Connected Components of the Sublevel and Superlevel Filtrations}\label{sec:connected_components}

\begin{proof}[Proof of \cref{prop:components_sublevel_even}]
First, note that each $H^\xi$ is a connected component of $\cup_{\xi \in \Xi_i} H^\xi$. To see this, we check that each $H^\xi$ is both open and closed in $\cup H^\xi$. $H^\xi$ is open as a finite intersection of open half-spaces. To check that $H^\xi$ is closed in $\cup H^\xi$, let $x^{(n)} \in H^\xi$ and $x^{(n)} \to x \in \cup H^\xi$ ($x \in \cup H^\xi$ is a crucial condition). Then $x_j \geq 0$ if $\xi_j > 0$ and $x_j \leq 0$ if $\xi_j < 0$. Since $x \in \cup H^\xi$, $x_j \neq 0$ for $j \geq i$. Therefore, $x \in H^\xi$. The closedness of $H^\xi$ in $\cup H^\xi$ then follows.

Now, for the first claim, it suffices to show that $\{x \in S^{D-1}: f_{k, D}(x) \leq t\} \subseteq \cup_{\xi \in \Xi_i} H^\xi$, because every connected set in $\cup_{\xi \in \Xi_i} H^\xi$, in particular, every connected component of $\{x \in S^{D-1}: f_{k, D}(x) \leq t\}$, lies in one of its components.

We show the contrapositive. Suppose $x \notin \cup_{\xi \in \Xi_i} H^\xi$, then $x_j = 0$ for some $j \geq i$, hence $x$ lies in the sub-sphere $S^{D-2} = \{x \in S^{D-1}: x_j = 0\}$. Applying \cref{thm:persistence_diagram_sublevel_even} (or by a direct computation) to $f_{k, D}$ restricted to $S^{D-2}$ shows that the global minimum of $f_{k, D}$ on $S^{D-2}$ is $\nu_{\{1, ..., D\} - \{j\}} \geq \nu_{\{1, ..., D\} - \{i\}} > t$. Then $f_{k,D}(x) > t$. The desired contrapositive, and hence the first claim then follows.

For the second claim, it suffices to show local minima points in each $H^\xi$ is connected by gradient flow lines in $\{x \in S^{D-1}: f_{k, D}(x) \leq t\}$, as all points in $\{x \in S^{D-1}: f_{k, D}(x) \leq t\}$ are connected to a critical point through a gradient descent flow, and each critical point is connected to a local minimum through a sequence of gradient descent flow lines (if the target is not a local minimum, there is a gradient descent flow originating from this target, the process iterates until a local minimum is reached). For the case $i = 1$, the fourth point of \cref{lem:critical_points_of_tensor} implies that each $H^\xi$ with $\xi \in \Xi_1$ contains exactly one local minimum (index-0 critical point) and no other critical point (because $\nu_{\{1, ..., D\} - \{1\}} < \nu_{\{1, ..., D\} - E}$ whenever $E$ is nonempty), the case for $i = 1$ then follows. Suppose the claim is true up to $i - 1$. Fix $H^\xi$. Then $H^{e_{i-1} + \xi} \cup H^{-e_{i-1} + \xi} \subseteq H^\xi$. Let $\zeta = \xi + \sum_{j < {i-1}} e_j$. Then $f_{k,D}(x^\zeta) = \nu_{\{1, ..., D\} - \{i-1\}}$. By \cref{prop:morse_differential_tensor}, it has a gradient descent flow line to $x^{\zeta + e_{i-1}} \in H^{e_{i-1} + \xi}$ and $x^{\zeta - e_{i-1}} \in H^{-e_{i-1} + \xi}$. Since the fourth point of \cref{lem:critical_points_of_tensor} implies there is a local minimum in $[\nu_{\{1, ..., D\} - \{i - 1\}}, \nu_{\{1, ..., D\} - \{i\}})$ for each $i > 1$, all local minima in $H^\xi$ are connected. The second claim then follows.
\end{proof}

\begin{proof}[Proof of \cref{prop:components_superlevel_even}]
This proposition is a direct corollary of the first and fourth point of \cref{lem:critical_points_of_tensor} and \cref{prop:morse_differential_tensor}. Specifically, there are gradient descent flow lines from $x^{\pm e_i}$ and $x^{e_1}$ to $x^{e_1 \pm e_i}$. Therefore, the merge of the components of $\pm e_i$ and $e_1$ must happen by $\nu_{\{1, i\}}$. The merge cannot happen earlier because otherwise \cref{cor:persistence_diagram_superlevel_even} is contradicted. The case for the connected component of $-e_1$ is similar.
\end{proof}

\begin{proof}[Proof of \cref{prop:components_sublevel_odd}]
Similar to that of \cref{prop:components_superlevel_even}.
\end{proof}

\section{Probabilistic Results}
\label{sec:probabilistic_results}

\begin{proof}[Proof of \cref{prop:LLN_sublevel_persistence_dim_0_even}]
Within this proof, let $\hat \nu = \nu_{\{1, ..., D\}}$ and $\hat \nu_i = \nu_{\{1, ..., D\} - \{i\}}$. Then $\mathcal{D}_0^\text{finite}$ consists of points of the form $(\hat \nu, \hat \nu_i)$, whose multiplicity is $2^{D-i}$.

Since
$$\frac{1}{D} (k\hat \nu)^{-\frac{2}{k-2}} = \frac{1}{D} \sum \lambda_i^{-\frac{2}{k-2}} \to I_k,$$
we have
\begin{align*}
k(\hat \nu_i - \hat \nu)
&= \left((k\hat \nu)^{-\frac{2}{k-2}} - \lambda_i^{-\frac{2}{k-2}}\right)^{-\frac{k-2}{2}} - k \hat \nu
\\&= k \hat \nu \left[ \left( 1 - (\frac{\lambda_i}{k\hat \nu})^{-\frac{2}{k-2}} \right)^{-\frac{k-2}{2}} - 1\right]
\\&= k\hat \nu \left[ \frac{k-2}{2} (\frac{\lambda_i}{k\hat \nu})^{-\frac{2}{k-2}} + O( (\frac{1}{\hat \nu} )^{-\frac{4}{k-2}}) \right]
\\&= (k\hat \nu)^\frac{k}{k-2} \frac{k-2}{2} \lambda_i^{-\frac{2}{k-2}} + O(\hat \nu^{\frac{k+2}{k-2}})
\\&= D^{-\frac{k}{2}} \left(\frac{k-2}{2} I_k^{-\frac{k}{2}} \lambda_i^{-\frac{2}{k-2}} + o(1) \right).
\end{align*}
Rearranging terms gives
$$D^{\frac{k}{2}}(\hat \nu_i - \hat \nu) = \frac{k-2}{2k} I_k^{-\frac{k}{2}} \lambda_i^{-\frac{2}{k-2}} + o(1)$$

We now verify the weak convergence. Let $\varphi: \mathbb{R} \to \mathbb{R}$ be a continuous function. Since $D^{k/2} (\hat \nu_i - \hat \nu)$ is bounded, without loss of generality, assume further that $\varphi$ is compactly supported and uniformly continuous.
Then 
\begin{align*}
\int \varphi d\mathcal{L}_D
&= \frac{1}{2^D} \sum_{1 \leq i \leq D} 2^{D-i} \varphi(D^{k/2} (\hat \nu_i - \hat \nu))
\\&= \sum_{1 \leq i \leq D} 2^{-i} \varphi(\frac{k-2}{2k} I_k^{-\frac{k}{2}} \lambda_i^{-\frac{2}{k-2}} + o(1)).
\\&= \sum_{1 \leq i \leq N} 2^{-i} \varphi(\frac{k-2}{2k} I_k^{-\frac{k}{2}} \lambda_i^{-\frac{2}{k-2}} + o(1)) + o(\varepsilon) & \text{ if } 2^{-N} < \varepsilon
\\&= \sum_{1 \leq i \leq N} 2^{-i} \varphi(\frac{k-2}{2k} I_k^{-\frac{k}{2}} \lambda_i^{-\frac{2}{k-2}}) + o(1) + O(\varepsilon) & \text{by uniform continuity}
\\&= \sum_{1 \leq i \leq N} 2^{-i} \varphi(\frac{k-2}{2k} I_k^{-\frac{k}{2}} \lambda_+^{-\frac{2}{k-2}}) + o(1) + O(\varepsilon) & \text{ for } i \leq N, \lambda_i \to \lambda_+ 
\\&= (1 - 2^{-N}) \varphi(\frac{k-2}{2k} I_k^{-\frac{k}{2}} \lambda_+^{-\frac{2}{k-2}}) + o(1) + O(\varepsilon)
\\&= \varphi(\frac{k-2}{2k} I_k^{-\frac{k}{2}} \lambda_+^{-\frac{2}{k-2}}) + o(1) + O(\varepsilon)
\end{align*}
The result then follows.
\end{proof}

\begin{proof}[Proof of \cref{prop:LLN_superlevel_persistence_dim_0_even}]
\begin{align}
\int \varphi d\mathcal{L}_D
&= \frac{1}{D} \left(\varphi(\nu_{\{1, 2\}}, \nu_{\{1\}}) + \sum_{2 \leq i \leq D} 2\varphi(\nu_{\{1, i\}}, \nu_{\{i\}})\right) \notag
\\&= \frac{2}{D} \sum_{1 \leq i \leq D} \varphi(\nu_{\{1, i\}}, \nu_{\{i\}}) + O(1/D) \notag
\\&= \frac{2}{D} \sum_{1 \leq i \leq D} \varphi\left(\left(\frac{1}{k} \left(\lambda_i^{-\frac{2}{k-2}} + \lambda_1^{-\frac{2}{k-2}}\right)^{-\frac{k-2}{2}}, \lambda_i/k \right)\right) + O(1/D) \notag
\\&= \frac{2}{D} \sum_{1 \leq i \leq D} \varphi\left(\left(\frac{1}{k} \left(\lambda_i^{-\frac{2}{k-2}} + \lambda_+^{-\frac{2}{k-2}}\right)^{-\frac{k-2}{2}}, \lambda_i/k \right)\right) + o(1) + O(1/D)  \label{eqn:superlevel_LLN_proof_line} 
\\&\to 2 \int_{\lambda_-}^{\lambda_+} \varphi\left(\left(\frac{1}{k} \left(\lambda^{-\frac{2}{k-2}} + \lambda_+^{-\frac{2}{k-2}}\right)^{-\frac{k-2}{2}}, \lambda/k \right) \right) \rho_{MP}(\lambda) d \lambda \notag
\\&= 2\int \varphi d(\nu^* \rho_{MP}(\lambda) d\lambda), \notag
\end{align}
where \cref{eqn:superlevel_LLN_proof_line} follows from the convergence $\lambda_1 \to \lambda^+$ and the uniform continuity of $(x,y) \mapsto \left(x^{-\frac{2}{k-2}} + y^{-\frac{2}{k-2}}\right)^{-\frac{k-2}{2}}$ for  $(x, y) \in [\frac{1}{2}\lambda_-, 2\lambda_+]^2$.
\end{proof}

\appendix

\section{Homological Algebra}
\label{sec:homological_algebra}

\begin{definition}[Chain Complex and its Boundary Homomorphism]\label{def:chain_complex}
A sequence of abelian groups (resp. vector spaces) $(C_q)_{q \in \mathbb{Z}}$ with homomorphisms (resp. linear maps) $\partial_q: C_q \to C_{q-1}$ is said to be a chain complex if 
\begin{equation}\label{eqn:boundary_squared}
\partial_q \partial_{q+1} = 0
\end{equation} for every $q$. The $\partial_q$'s are called the boundary homomorphisms of $(C_q)$. The groups (resp. vector spaces) $C_q$'s are called chain groups (resp. vector spaces). Elements of chain groups are called chains.
\end{definition}

\begin{remark}$\quad$
\begin{itemize}
\item We often drop the boundary homomorphism in the notation of $(C_q)$ to simplify notation.
\item We often denote $(C_q)$ by $C_\bullet$ to emphasize that we are referring to the whole complex rather than an individual group in the complex.
\item When constructing chain complexes, we often just define the chain groups for certain indices, say nonnegative integers, rather than all integers. Unless otherwise specified, the undefined groups are understood to be $0$.
\end{itemize}
\end{remark}

\cref{eqn:boundary_squared} ensures that $\im \partial_{q+1} \subseteq \ker \partial_q$. \emph{Homology} quantifies the extent to which this inclusion fails to be an equality.

\begin{definition}[Homology]
Let $(C_q)$ be a chain complex with boundary homomorphisms $\partial_q$'s.
For each $q \geq 0$, the $q^{th}$ homology group $H_q(C_\bullet)$ of $(C_q)$ is the quotient group
$$H_q(C_\bullet) = \ker \partial_q / \im \partial_{q+1},$$
\end{definition}

\begin{remark}
We often denote $(H_q(C_\bullet))$ by $H_*(C_\bullet)$ to emphasize that we are referring to the whole sequence of homology groups rather than an individual homology group.
\end{remark}

Denote by $\Delta^q$ the standard $q$-dimensional simplex, i.e. the convex hull of the standard basis $\{e_0, ..., e_q\}$ in $\mathbb{R}^{q+1}$.

\begin{definition}[Singular Homology with Coefficients]\label{def:singular_homology}
Let $X$ be a topological space and $G$ be an abelian group. The singular chain complex $C_\bullet(X; G)$ with coefficients in $G$ is defined as follows.
\begin{itemize}
\item For each $q \geq 0$, the $q^{th}$ singular chain group is $$C_q(X; G) = \{\sum_{1 \leq i \leq N} g_i \sigma_i : g_i \in G, N \in \mathbb{N}_0, \sigma_i \text{ is a continuous map from $\Delta^q$ to $X$}\}.$$
\item For $q \geq 1$, the boundary homomorphism $\partial_q: C_q(X; G) \to C_{q-1}(X; G)$ is defined by
$$\partial_q \sigma = \sum_{0 \leq i \leq q} (-1)^i \sigma \circ \Phi_{e_0, ..., \hat e_i, ..., e_q},$$
where $\circ$ means function composition, and $\Phi_{e_0, ..., \hat e_i, ..., e_q}: \Delta^{q-1} \to \Delta^q$ is the affine map that sends $e_0, ..., e_{i-1}$ to $e_0, ..., e_{i-1}$ and $e_{i+1}, ..., e_{q-1}$ to $e_i, ..., e_q$.
\end{itemize}
The singular homology groups of $X$ with coefficients in $G$ are those of $C_\bullet(X; G)$.

The elements of singular chain groups are called singular chains.
\end{definition}

\begin{remark}$\quad$
\begin{enumerate}
\item Note that if $G$ is a field, the singular chain groups and homology groups are vector spaces over the field.
\item Even though each $C_\bullet(X; G)$ often consists of groups with infinite rank (or infinite-dimensional vector spaces if $G$ is a field), $H_*(X; G)$ often has finite rank (or dimension) in practice.
\end{enumerate}
\end{remark}

\begin{definition}[Induced maps at the level of chains and homology]
Let $f: X \to Y$ be a continuous function and $G$ be an abelian group. For each $q$, the induced map $f_\#: C_q(X; G) \to C_q(Y; G)$ on the chain level is defined by mapping each $\sigma: \Delta^q \to X$ to $f \circ \sigma$, where $\circ$ denotes composition. The induced map $f_*: H_q(X; G) \to H_q(Y; G)$ at the level of homology is defined by $f_*(z + \im \partial^X_{q+1}) = f_\# z + \im \partial^Y_{q+1}$, where $\partial^X_{q+1}$ and $\partial^Y_{q+1}$ are the boundary homomorphisms at dimension $q+1$ of the singular chain complexes of $X$ and $Y$ respectively.
\end{definition}

Induced homomorphisms are functorial, in the sense that
\begin{itemize}
\item $(\id_{X})_q = \id_{H_q(X)}$ for every topological space $X$, where $\id$ denotes the identity map or homomorphism, and
\item $(gf)_* = g_* f_*$ for every pair of continuous functions $f: X \to Y$ and $g: Y \to Z$.
\end{itemize}
(Cf. 
items (i) and (ii) on p.111 of \citep{hatcher02_algtopo}
).

The most commonly considered induced maps are those induced by inclusion:
\begin{definition}[Inclusion Map]
Let $Y$ be a topological space and $X \subseteq Y$. The inclusion map $i: X \to Y$ is defined by $i(x) = x$.
\end{definition}

Related to the notion of functoriality is the phrase ``commutativity". In homological algebra, a \emph{diagram} is a directed graph where each node is a set (usually groups or vector spaces) and each edge is a function (usually homomorphisms or linear maps) from the source set to the target set. A diagram is said to \emph{commute} if the composition of the functions along paths with the same source and target are equal. For example, consider the diagram
$$
\begin{tikzcd}
\{0\} \arrow[d, "0 \mapsto 4"] \arrow[r, "0 \mapsto 1"] & \{1, 2\} \arrow[d, "0"]
\\
\{3, 4\} \arrow[r, "0"] & \{0, 1\}                                 
\end{tikzcd},
$$
where the arrows annotated by $0$ correspond to the constantly 0 function. This diagram commutes, because applying the top horizontal map and then the right vertical map is the same as applying the left vertical map and then the bottom horizontal map (both maps $0$ to $0$). On the other hand, the diagram
$$
\begin{tikzcd}
\{0\} \arrow[d, "0 \mapsto 4"] \arrow[r, "0 \mapsto 1"] & \{1, 2\} \arrow[d, "0"]
\\
\{3, 4\} \arrow[r, "x \mapsto x - 3"] & \{0, 1\}                                   
\end{tikzcd}
$$
which differs from the previous diagram only by the bottom map, does \emph{not} commute, because the composition of the top and right maps sends $0$ to $0$ and the composition of the left and bottom maps sends $0$ to $1$.

Now, we discuss reduced homology, which is needed for the statement of Alexander duality. We give two equivalent definitions. One shows it is well-defined (in particular, it is independent of the choice of base points), the other is the operational one.

\begin{definition}[Reduced Homology, p.110 of \citep{hatcher02_algtopo}]\label{def:reduced_homology}
Consider a chain complex $C_\bullet$ with coefficients in an abelian group $G$. Suppose $C_q = 0$ for $q < 0$ and $C_0$ has a $G$-basis $\beta$, i.e. every element can be written as a unique finite linear combination of elements in this basis with coefficients in $G$. Define ${\tilde C}_\bullet$ as follows.
\begin{itemize}
\item For $q \neq -1$, ${\tilde C}_q = C_q$. For $q \notin \{0, -1\}$, ${\tilde \partial}_q = \partial_q$.
\item $\tilde C_{-1} = G$, and
$${\tilde \partial}_0 \sum g_i \sigma_i = \sum g_i$$
for $\sigma_i \in \beta$, and ${\tilde \partial}_{-1} = 0$.
\end{itemize}
If $\tilde \partial_0 \partial_1 = 0$, then ${\tilde C}_\bullet$ is said to be the reduced complex of $C_\bullet$, and the reduced homology of $C_\bullet$ is the homology of ${\tilde C}_\bullet$.
\end{definition}

It can be shown that for the singular complex and the Morse complex of a nonempty space, the reduced homology is well-defined. In this case,
\begin{equation*}
\tilde H_q(C_\bullet; G) \cong \begin{cases}
H_q(C_\bullet; G) & \text{ if } q \neq 0
\\
H_0(C_\bullet; G) \oplus G & \text{ if } q = 0
\end{cases}
\end{equation*}

\begin{definition}[Pair of Topological Spaces and Relative Homology]$\quad$
\begin{itemize}
\item A two-tuple of topological spaces $(X, A)$ is said to be a pair if $A$ is a subspace of $X$. We identify a space $X$ with the pair $(X, \emptyset)$, where $\emptyset$ denotes the empty set.
\item A continuous map $f: (X, A) \to (Y, B)$ is a continuous map $f: X \to Y$ such that $f(A) \subseteq B$.
\item The relative homology $H_*(X, A; G)$ with coefficients in an abelian group $G$ is the homology of the chain complex with groups $C_q(X, A; G) = C_q(X; G)/C_q(A; G)$ and boundary homomorphisms induced by those of $C_\bullet(X)$.
\end{itemize}
\end{definition}

\begin{proposition}[Reduced Homology as Relative Homology, Example 2.18 of \citep{hatcher02_algtopo}] Let $X$ be a nonempty topological space and $x_0 \in X$. Let $G$ be an abelian group. Then $\tilde H_q(X; G) \cong H_q(X, x_0; G)$.
\end{proposition}

\begin{definition}[Persistent Reduced Homology]
Let $f: X \to \mathbb{R}$ be a continuous function that attains its infimum and $G$ be an abelian group. Let $x_0$ be a global minimum of $f$. The persistent reduced homology of the sublevel filtration of $f$ with coefficients in $G$ is the persistence module where each group is the reduced homology of $f^{-1}[t, \infty)$, where $t$ ranges from $[\min f, \infty)$. It is isomorphic to the persistence module of $H_q(f^{-1}[t, \infty), x_0)$.
\end{definition}

\begin{definition}[Cohomology]\label{def:cohomology}
Let $C_\bullet$ be a chain complex and $G$ be an abelian group. Define the (reversed) chain complex $C^\bullet(G)$ as follows.(Note the superscript in the notation.) Let $C^q(G)$ be the group of homomorphisms from $C_q$ to $G$. The addition in $C^q(G)$ is pointwise addition. Define $\partial^q: C^q \to C^{q+1}$ by $\partial^q \varphi = \varphi \circ \partial_{q+1}$. Then $\partial^q \partial^{q-1} = 0$. The cohomology group with coefficients in $G$ is defined by $H^q(C^\bullet; G) = \ker \partial^q / \im \partial^{q-1}$. Elements of $C^q(G)$ are called cochains. The singular cohomology of a topological space $X$ with coefficients in $G$ is the cohomology of the singular chain complex of $X$ with coefficients in $G$.
\end{definition}

\begin{definition}[Induced Maps for Cohomology]
Let $f: X \to Y$ be a continuous function and $G$ be an abelian group. For each $q$, the induced map $f^\#: C^q(Y; G) \to C^q(X; G)$ (note the reversed direction) on the chain level is defined by $f^\# \varphi = \varphi \circ f_\#$, where $\circ$ denotes composition. The induced map $f^*: H^q(Y; G) \to H^q(X; G)$ (note the reversed direction) at the level of homology is defined by $f^*(z + \im \partial_X^{q+1}) = f^\# z + \im \partial_Y^{q+1}$, where $\partial_X^{q+1}$ and $\partial_Y^{q+1}$ are the boundary homomorphisms at dimension $q+1$ of the singular cochain complexes of $X$ and $Y$ respectively.
\end{definition}
Everything is analogous to the homology case with directions reversed. For instance, one can define reduced cohomology and persistent cohomology.

We conclude by addressing subtle issues about persistent homology.

\begin{definition}[Filtered Chain Complex]\label{def:filtered_chain_complex}
Let $C_\bullet$ be a chain complex with coefficients in $\mathbb{F}$ with bases $\beta_\bullet$. Let $f: \cup_q \beta_q \to \mathbb{R}$. $C_\bullet$ is said to be filtered by $f$ if $f(c_i) \leq f(c)$ whenever $\partial c = \sum a_i c_i$, where $a_i \in \mathbb{F} - \{0\}$ and $c_i \in \cup_q \beta_q$.
\end{definition}

\begin{proof}[Proof of \cref{lem:persistence_diagram_reduced_homology}]
The first claim follows from the fact that homology and reduced homology coincide at every positive dimension. For the second claim, fix the coefficient field $\mathbb{F}$ and let $x_0$ be a global minimum of $f$. We have the following segment of long exact sequence (Cf. Theorem 2.16 of \citep{hatcher02_algtopo} and 2 paragraphs after its proof):
$$H_0(\{x_0\}) \to H_0(X_t) \to H_0(X_t, \{x_0\}) \to 0.$$
Recall from Proposition 2.7 of \citep{hatcher02_algtopo} that $H_0(X_t)$ is a direct sum of $\mathbb{F}$’s, one for each path-component of $X_t$. Then $H_0(\{x_0\}) \cong \mathbb{F}$ is generated by the component of $x_0$ (which is $\{x_0\}$ itself), and the first map sends this component to the component of $x_0$ in $X_t$.
Since the first map is nonzero, it is injective. In other words, we have the following short exact sequence:
$$0 \to H_0(\{x_0\}) \to H_0(X_t) \to H_0(X_t, \{x_0\}) \to 0.$$
Further, this sequence is natural in the sense that for $s \leq t$, the following diagram commutes:
$$
\begin{tikzcd}
0 \arrow[r] & H_0(\{x_0\}) \arrow[r, "i"] \arrow[d] & H_0(X_s) \arrow[r] \arrow[d] & {H_0(X_s, \{x_0\})} \arrow[r] \arrow[d] & 0 \\
0 \arrow[r] & H_0(\{x_0\}) \arrow[r, "i"]           & H_0(X_t) \arrow[r]           & {H_0(X_t, \{x_0\})} \arrow[r]           & 0
\end{tikzcd}
$$
We claim that $H_0(X_t) \cong H_0(\{x_0\}) \oplus H_0(X_t, \{x_0\})$, and this isomorphism extends to the level of persistence module. By the splitting lemma, it suffices to construct maps $p_t$ such that $p_t i$ is the identity on $H_0(\{x_0\})$ and the following diagram commutes
$$
\begin{tikzcd}
0 \arrow[r] & H_0(\{x_0\}) \arrow[d] & H_0(X_s) \arrow[l, "p_s"']\arrow[r] \arrow[d] & {H_0(X_s, \{x_0\})} \arrow[r] \arrow[d] & 0 \\
0 \arrow[r] & H_0(\{x_0\})           & H_0(X_t) \arrow[l, "p_t"'] \arrow[r]           & {H_0(X_t, \{x_0\})} \arrow[r]           & 0
\end{tikzcd}
$$
Define $p_t$ by sending the component of $x_0$ in $X_t$ to the component of $x_0$ in $\{x_0\}$, and all other components to $0$. Then it is straight-forward to check that $p_t$ satisfies the above conditions. Therefore, as a persistence module, $H_0(X_t) \cong H_0(\{x_0\}) \oplus H_0(X_t, \{x_0\})$. Note that $H_0(X_t)$ is an $(\min f, \infty)$-interval module. The lemma then follows.
\end{proof}

\section{The Smale Condition}

In this section, we check that the gradient field of $f_{k,D}$ satisfies the Smale condition on the unit sphere. This section was drafted by Harmonic's Aristotle and Opus 5.0, and the argument is curated by the author.

We first recall the relevant terminology.

The stable and unstable manifolds, denoted by $W^s(x)$ and $W^u(x)$, of a critical point $x$ of a gradient descent flow $\varphi^t(x')$ (where $t$ denotes time and $x'$ denotes initial location) are defined as follows (Section 2.1d of \citep{audin14_morseTheory}).
\begin{align*}
W^s(x) &= \{x': \lim_{t \to +\infty} \varphi^t(x') = x\}\\
W^u(x) &= \{x': \lim_{t \to -\infty} \varphi^t(x') = x\}
\end{align*}

Two submanifolds $M$ and $N$ of a manifold $S$ are said to be transverse if for every $x \in M \cap N$, the tangent space of $S$ at $x$ is spanned by the tangent spaces of $M$ and $N$ at $x$, or symbolically, $T_x S = T_x M + T_x N$ (Section A.3 of \citep{audin14_morseTheory}).

\begin{definition}[Smale Condition (Section 2.2b of \citep{audin14_morseTheory})]\label{def:smale_condition}
A smooth function $f$ on a manifold $S$ is said to satisfy the Smale condition if the stable and unstable manifolds of its gradient descent flow are transverse.
\end{definition}

To show that $f_{k,D}$ satisfies the Smale condition, we first explicitly identify the stable and unstable manifolds. Recall from \cref{lem:never_descent_to_zero} that we have the following negative gradient flow.
\begin{align}
x' &= -[\nabla f_{k,D}(x) - (\nabla f_{k,D}(x) \cdot x)x], \notag\\
\intertext{or in coordinate form,}
x_i' &= -(\lambda_i x_i^{k-2} - k f_{k,D}(x)) x_i \text{ for each } i \label{eqn:negative_gradient_flow_coordinates}
\end{align}
Let
\begin{align*}
S_\xi &= \{x \in S^{D-1} : x_i \text{ and } \xi_i \text{ have the same sign strictly.}\}\\
T_\xi &= \{x \in S^{D-1} :
\text{ For } \xi_i \neq 0, x_i \text{ and } \xi_i \text{ have the same sign strictly and } |x_i| \propto \lambda_i^{-1/(k-2)};\\
&\qquad \qquad \qquad \qquad \text{ for }\xi_j = 0 \text{ and } \xi_i \neq 0, \xi^k (\lambda_i x_i^{k-2} - \lambda_j x_j^{k-2}) > 0
\}
\end{align*}

In the above and hereafter, the strict sign of a number $a$ is $1$, $0$, or $-1$ depending on whether $a$ is strictly positive, zero, or strictly negative.

\begin{proposition}\label{prop:stable_unstable_manifold_identification}
Let $x^\xi$ be a critical point of $f_{k,D}$.
\begin{enumerate}
\item If $\xi^k = 1$, then $W^s(x^\xi) = S_\xi$ and $W^u(x^\xi) = T_\xi$.
\item If $\xi^k = -1$, then $W^s(x^\xi) = T_\xi$ and $W^u(x^\xi) = S_\xi$.
\end{enumerate}
\end{proposition}

To prove this, we first establish two lemmas.

\begin{lemma} \label{lem:sign_preservation}
Along the negative gradient flow trajectory, 
\begin{enumerate}
\item the strict sign of $x_i$ is preserved for each $i$; and
\item the strict sign of $\lambda_i x_i^{k-2} - \lambda_j x_j^{k-2}$ is preserved for each pair of $i, j$.
\end{enumerate}
Further, suppose that as $t \to -\infty$, $x \to x^\zeta$, with $\zeta^k = 1$, $\zeta_i \neq 0, \zeta_j \neq 0$. Then $\lambda_i x_i^{k-2}(t) \equiv \lambda_j x_j^{k-2}(t)$.
\end{lemma}

\begin{proof}
The first claim follows from the linearity of \cref{eqn:negative_gradient_flow_coordinates}. For the second claim, let $y_i = \lambda_i x_i^{k-2}$.
\begin{align*}
y_i' &= (k-2)\lambda_i x_i^{k-3} x_i' 
\\&= -(k-2)\lambda_i x_i^{k-3}(\lambda_i x_i^{k-2} - kf_{k,D}(x)) x_i 
\\&= -(k-2)y_i(y_i - kf_{k,D}(x)).
\\\intertext{Hence,}
(y_i - y_j)' &= -(k-2)(y_i - y_j)(y_i + y_j - kf_{k,D}(x)).
\end{align*}
The second claim then follows from linearity of the ordinary differential equation.

For the last claim, note that at $x^\xi$, $y_i + y_j - kf_{k,D}(x) = \left ( \sum_{\zeta_\ell \neq 0} \lambda_\ell^{-2/(k-2)} \right)^{1 - k/2} > 0$. Hence the differential equation implies that $y_i - y_j$ is either identically 0 or it decays exponentially as time moves forward (equivalently, explodes exponentially as time moves backward). Since $y_i - y_j$ is bounded (say, by $2\lambda_1$), the difference is identically 0.
\end{proof}

\begin{lemma}\label{lem:absolute_coordinates_decreases}
Suppose $\xi^k = 1$ and $\xi_i \neq 0$. Let $x \in T_\xi$. Then $\lambda_i x_i^{k-2} - kf_{k,D}(x) \geq 0$. Further, by the ordinary differential equation \cref{eqn:negative_gradient_flow_coordinates}, along the trajectory of $x$, $|x_i|$ is nonincreasing.
\end{lemma}

\begin{proof}
Observe that by the definition of $T_\xi$, whenever $\xi_\alpha, \xi_\beta \neq 0$ and $\xi_\gamma = 0$, $\lambda_\alpha x_\alpha^{k-2} = \lambda_\beta x_\beta^{k-2} > \lambda_\gamma x_\gamma^{k-2}$. Furthermore, these relations hold along the trajectory by \cref{lem:sign_preservation}. Denote by $C(x)$ the common values of $\lambda_\alpha x_\alpha^{k-2}$ for $\xi_\alpha \neq 0$. Hence at $x$ and along its trajectory,
\begin{align*}
\lambda_i x_i^{k-2} - kf_{k,D}(x)
&= C(x) - \sum \lambda_\ell x_\ell^k
\\&= C(x) - \sum (\lambda_\ell x_\ell^{k-2}) x_\ell^2
\\&= C(x) - \left(\sum_{\xi_\ell \neq 0} C(x) x_\ell^2 + \sum_{\xi_\gamma = 0} (\lambda_\gamma x_\gamma^{k-2}) x_\gamma^2 \right)
\\& \geq C(x) - \left(\sum_{\xi_\ell \neq 0} C(x) x_\ell^2 + \sum_{\xi_\gamma = 0} C(x) x_\gamma^2 \right) 
\\& = C(x) - C(x) \sum x_\ell^2
\\& = 0.
\end{align*}
\end{proof}

\begin{proof}[Proof of \cref{prop:stable_unstable_manifold_identification}]
First, we argue that it suffices to consider the case $\xi^k = 1$. Indeed, for $\xi^k = -1$, $k$ is odd and all nonzero $x_i$'s are strictly negative. The case for $\xi^k = 1$ implies that $W^u(x^{-\xi}) = T_{-\xi} = -T_\xi$. Since $f_{k,D}(-x) = -f_{k,D}(x)$, $W^s(x^\xi) = -W^u(x^{-\xi})$. The second claim then follows. The case for the stable manifold is similar.

Now, fix $\xi$ and suppose $\xi^k = 1$. To show $W^s(x^\xi) \subseteq S_\xi$, let $x \in W^s(x^\xi)$. Then $x^\xi = \lim_{t \to +\infty} x(t)$. By a limiting argument, if $x^\xi_i \neq 0$, then $x_i \neq 0$. If $x^\xi_i = 0$, by the second point of \cref{lem:never_descent_to_zero}, $x_i = 0$.

The argument for $S_\xi \subseteq W^s(x^\xi)$ is similar. Let $x \in S_\xi$ and $x^\zeta = \lim_{t \to +\infty} x(t)$. To apply \cref{lem:never_descent_to_zero}, we need to argue that $\zeta^k = 1$: Suppose not, then $k$ is odd and $\zeta_i < 0$ for some $i$. A limiting argument shows that $x_i < 0$, and hence $\xi^k = \xi_i^k = \xi_i = -1$, contradictory to our assumption that $\xi^k = 1$. The rest of the argument is analogous to the reverse inclusion.


To show $W^u(x^\xi) \subseteq T_\xi$, suppose $x(0) = x$ and $\lim_{t \to -\infty} x(t) = x^\xi$. Let $\xi_i, \xi_j \neq 0$. At $x^\xi$, this difference is 0. Then $\lambda_i x_i^{k-2} - \lambda_j x_j^{k-2} = 0$ by the last claim of \cref{lem:sign_preservation}. Therefore for $\xi_i \neq 0$, $\lambda_i x_i^{k-2} = C$ that is independent of $i$; in other words, $|x_i| \propto \lambda_i^{-1/(k-2)}$. For $\xi_j = 0$ and $\xi_i \neq 0$, $\lambda_i x_i^{k-2} - \lambda_j x_j^{k-2}$ converges to a positive number ($\lambda_i (x^\xi_i)^{k-2}$), since the strict sign is conserved, the difference is strictly positive throughout.

To show $T_\xi \subseteq W^u(x^\xi)$, let $x \in T_\xi$ and $x^\zeta = \lim_{t \to -\infty} x(t)$. It suffices to show that $\zeta_i = 0$ if and only if $\xi_i = 0$. Indeed, this implies that $\zeta_i$ and $\xi_i$ have the same strict sign, because whenever $\zeta_i \neq 0$, it has the same strict sign as $x_i$, and hence as $\xi_i$.

If $\xi_i \neq 0$, then $x_i \neq 0$. By \cref{lem:absolute_coordinates_decreases}, $|x_i|$ is non increasing as time moves forward along the negative gradient flow, hence $|x^\zeta_i| \geq |x_i| > 0$, i.e. $\zeta_i \neq 0$.

We show that $\zeta_j \neq 0$ implies $\xi_j \neq 0$ by contradiction. Suppose $\zeta_j \neq 0$ and $\xi_j = 0$. Since $\xi \neq 0$, let $\xi_i \neq 0$. Then $x_i \neq 0$. Again, by \cref{lem:absolute_coordinates_decreases}, $\zeta_i \neq 0$.

We next establish that $\zeta^k = 1$: Suppose not. Then $k$ is odd and $\zeta_i = -1$. Since sign is preserved, $x_i < 0$. By the definition of $T_\xi$, $\xi_i$ has the same sign as $x_i$, and hence $\xi^k = \xi_i^k = \xi_i = -1$, contradictory to our assumption that $\xi^k = 1$.

By the definition of $T_\xi$, since $\xi_i \neq 0 = \xi_j$, $\lambda_i x_i^{k-2} - \lambda_j x_j^{k-2} > 0$. This, however, contradicts with the last claim of \cref{lem:sign_preservation}, as desired.

\end{proof}

\begin{proposition} $\quad$

\begin{enumerate}
\item The tangent space of $S_\xi$ at $x$ is $x^\perp \cap \{v \in \mathbb{R}^D: v_i = 0 \text{ whenever }\xi_i = 0\}$.
\item The tangent space of $T_\xi$ at $x$ is $x^\perp \cap \{v \in \mathbb{R}^D: \lambda_i x_i^{k-3} v_i = \lambda_j x_j^{k-3} v_j \text{ whenever } \xi_i \neq 0, \xi_j \neq 0\}$, and its dimension is $D - |\{i: \xi_i \neq 0\}|$.
\end{enumerate}
\end{proposition}

\begin{proof}
The first claim is trivial. For the second claim, we recall basic manifold theory.

\begin{proposition}[Corollary 5.14 of \citep{lee_smooth_manifold}]
Every regular level set of a smooth map between smooth manifolds is a properly embedded submanifold.
\end{proposition}

\begin{proposition}[Proposition 5.38 of \citep{lee_smooth_manifold}]
Suppose $S$ is a smooth manifold and $M \subseteq S$ is an embedded submanifold. Let $p \in S$ and $U$ be a neighborhood of $p$ in $S$. If $M \cap U$ is locally the regular level set of some smooth $F: U \to N$, then the tangent space of $M$ at a point $p$ is the kernel of $dF$ at $p$. 
\end{proposition}

To construct the $U$ and $F$ in the proposition, fix $x \in T_\xi$ and $i_0$ such that $\xi_{i_0} \neq 0$. Let 
$$U = \{\tilde x \in S^{D-1}: \xi_i \text{ and } \tilde x_i \text{ have the same sign when }\xi_i \neq 0, \quad \xi^k(\lambda_{i_0} \tilde x_{i_0}^{k-2} - \lambda_j \tilde x_j^{k-2}) > 0 \text{ for } \xi_j = 0\}.$$
Then $U$ is an open neighborhood of $x$ in $S^{D-1}$. Define $F: U \to \mathbb{R}^{|\{i: \xi_i \neq 0\}| - 1}$ by
$F_\alpha(x) = \lambda_{i_\alpha} x_{i_\alpha}^{k-2} - \lambda_{i_0} x_{i_0}^{k-2}$, where the $i_\alpha$'s range over $\{i: \xi_i \neq 0, i \neq i_0\}$. Then $T_\xi = F^{-1}(0)$ and in ambient coordinates,
$$[dF_x(v)]_\alpha = (k-2) (\lambda_{i_\alpha} x_{i_\alpha}^{k-3} v_{i_\alpha} - \lambda_{i_0} x_{i_0}^{k-3} v_{i_0}).$$
Direct computation shows that the kernel of $dF$ at $x$ is indeed the expression in the proposition.

It remains to show $x$ is a regular point, i.e. $dF_x$ is surjective. (The dimension claim follows from the rank-nullity theorem.) Consider the following matrix. Its first row is $dF_x$ in matrix form (scaled by $(k-2)$), and its second row enforces the tangency condition on the sphere. It suffices to show this matrix is surjective.
\[
\begin{blockarray}{cccc}
& \{i: \xi_i \neq 0, i \neq i_0\} & i_0 & \{j: \xi_j = 0\}  \\
\begin{block}{c[ccc]}
\{i: \xi_i \neq 0, i \neq i_0\}    & D & -\lambda_{i_0} x_{i_0}^{k-3} \mathbf{1} & 0 \\
\text{tangency to sphere} & a^T & x_{i_0} & x^T|\{j: \xi_j = 0\} \\
\end{block}
\end{blockarray},
\]
where $D$ is a diagonal matrix with entries $\lambda_{i_\alpha} x_{i_\alpha}^{k-3}$ (we abuse notation here as $D$ also denotes the ambient dimension), and $a = x|\{i: \xi_i \neq 0, i \neq i_0\}$.
The surjectivity of the above matrix is a consequence of the invertibility of the following square matrix.
$$
\begin{bmatrix}
D & -\lambda_{i_0} x_{i_0}^{k-3} \mathbf{1}
\\
a^T & x_{i_0}
\end{bmatrix}
$$
Row reducing gives the following upper triangular matrix
$$
\begin{bmatrix}
I & 0
\\
- a^T D^{-1} & 1
\end{bmatrix}
\begin{bmatrix}
D & -\lambda_{i_0} x_{i_0}^{k-3} \mathbf{1}
\\
a^T & x_{i_0}
\end{bmatrix}
=
\begin{bmatrix}
D & -\lambda_{i_0} x_{i_0}^{k-3} \mathbf{1}
\\
0 & x_{i_0}  + \lambda_{i_0} x_{i_0}^{k-3} (a^T D^{-1}) \mathbf{1}
\end{bmatrix}.
$$
It suffices to show that $x_{i_0}  + \lambda_{i_0} x_{i_0}^{k-3} (a^T D^{-1}) \mathbf{1} \neq 0$. Recall that, for $\xi_i \neq 0$, $\lambda_i x_i^{k-2} = C(x)$, where $C(x)$ does not depend on $i$.
\begin{align*}
x_{i_0}  + \lambda_{i_0} x_{i_0}^{k-3} (a^T D^{-1}) \mathbf{1}
&= x_{i_0} + \lambda_{i_0} x_{i_0}^{k-3} \sum_{\substack{\xi_j \neq 0\\j \neq i_0}} x_j/(\lambda_j x_j^{k-3})
\\&= x_{i_0} + C(x)/x_{i_0} \sum_{\substack{\xi_j \neq 0\\j \neq i_0}} x_j^2 /C(x)
\\&= \frac{1}{x_{i_0}} \left ( x_{i_0}^2  + \sum_{\substack{\xi_j \neq 0\\j \neq i_0}} x_j^2 \right)
\\&= \frac{1}{x_{i_0}} \left (\sum_{\xi_j \neq 0} x_j^2 \right)
\\& \neq 0.
\end{align*}
The proposition then follows.
\end{proof}

\begin{theorem}\label{thm:f_is_smale}
$f_{k,D}$ satisfies the Smale condition, i.e. whenever $x \in W^u(x^\xi) \cap W^s(x^\zeta)$, the tangent space of the unit sphere at $x$ is spanned by those of $W^u(x^\xi)$ and $W^s(x^\zeta)$ at $x$.
\end{theorem}

\begin{proof}
There are four possibilities.
\begin{enumerate}
\item $\xi^k = 1$, $\zeta^k = 1$, $x \in T_\xi \cap S_\zeta$.
\item $\xi^k = -1$, $\zeta^k = 1$, $x \in S_\xi \cap S_\zeta$.
\item $\xi^k = 1$, $\zeta^k = -1$, $x \in T_\xi \cap T_\zeta$. 
\item $\xi^k = -1$, $\zeta^k = -1$, $x \in S_\xi \cap T_\zeta$.
\end{enumerate}

For the first case, by the second claim of \cref{lem:never_descent_to_zero}, $\xi_i = 0$ and $x_i = 0$ whenever $\zeta_i = 0$. Pick any $v \in x^\perp$. Let $v^\zeta$ be the projection to $\text{span } \{e_i: \zeta_i \neq 0\}$. $v_\zeta$ is still tangent to the sphere because
$$\sum v^\zeta_i x_i = \sum v_i x_i - \sum_{\zeta_i = 0} v_i x_i = 0 - 0 = 0 \quad \text{(since $\zeta_i = 0$ implies $x_i = 0$)}.$$
Let $v^\xi = v - v^\zeta$. Then whenever $\xi_i \neq 0$, $\zeta_i \neq 0$, and hence $v^\xi_i = 0$, and hence $\lambda_i x_i^{k-3} v^\xi_i = 0$. Therefore, $v^\xi$ lies in the tangent space of $T_\xi$.

The fourth case is analogous.

The second case is vacuously true because the intersection is empty. Indeed, the third point of \cref{lem:never_descent_to_zero} implies whenever $\zeta_i$ is nonzero, then so is $\xi_i$. The first point of \cref{lem:never_descent_to_zero} then implies that whenever $\zeta_i$ is nonzero, $\xi_i$ and $\zeta_i$ have the same sign. Since $\xi^k = -1$, $k$ is odd. Therefore $\xi_i$ is nonpositive and $\zeta_i$ is nonnegative. These together imply that $\zeta_i \equiv 0$, a contradiction.

For the third case we appeal to a dimension argument. We have shown that dimension of the tangent space of $T_\xi$ is $D - |\{i: \xi_i \neq 0\}|$. Similarly the dimension of $T_\zeta$ is $D - |\{i: \zeta_i \neq 0\}|$. Since $\xi^k \neq \zeta^k$, $\{i: \zeta_i \neq 0\}$ and $\{i: \xi_i \neq 0\}$ are disjoint, hence the intersection of the two tangent spaces are defined by $|\{i: \xi_i \neq 0\}| - 1 + |\{i: \zeta_i \neq 0\}| - 1 + 1 = |\{i: \xi_i \neq 0\}| + |\{i: \zeta_i \neq 0\}| - 1$ equations, where the equations from the two spaces involve disjoint variables. Therefore, the intersection of the two tangent spaces is the null space of the following matrix
\[
\begin{blockarray}{cccccc}
& \{i: \xi_i \neq 0, i \neq i_0\} & \{j: \zeta_j \neq 0, j \neq j_0\} & i_0 & j_0 & \{\ell: \xi_\ell = \zeta_\ell = 0\}  \\
\begin{block}{c[ccccc]}
\{i: \xi_i \neq 0, i \neq i_0\} & D_1 & & -\lambda_{i_0} x_{i_0}^{k-3} \mathbf{1} \\
\{j: \zeta_j \neq 0, j \neq j_0\} & & D_2 & & -\lambda_{j_0} x_{j_0}^{k-3} \mathbf{1} \\
\text{tangency to sphere} & x^T|\{i: \xi_i \neq 0, i \neq i_0\} & x^T|\{j: \zeta_j \neq 0, j \neq j_0\} & x_{i_0} & x_{j_0} & x^T|\{\ell: \xi_\ell = \zeta_\ell = 0\} \\
\end{block}
\end{blockarray},
\]
where $D_1$ and $D_2$ are analogous diagonal matrices like $D$ in the previous proof. Row reducing this matrix as in the previous proof shows that this matrix is also full-rank, and hence the dimension of the intersection is $D - (|\{i: \xi_i \neq 0\}|-1 + |\{j: \zeta_j \neq 0\}|-1 + 1) = D - (|\{i: \xi_i \neq 0\}| + |\{j: \zeta_j \neq 0\}|) + 1$. Combining, the dimension of the span of the two tangent spaces is 
$$(D - |\{i: \xi_i \neq 0\}|) + (D - |\{i: \zeta_i \neq 0\}|) - [D - (|\{i: \xi_i \neq 0\}| + |\{i: \zeta_i \neq 0\}|) + 1] = D - 1,$$
as desired.

\end{proof}


\bibliography{bib_positive_orthogonally_decomposable_tensor.bib}
\end{document}